\documentclass[12pt]{amsart}
\usepackage{etoolbox}

\makeatother

\usepackage[utf8]{inputenc}
\usepackage{graphicx}
\graphicspath{ {figures/} }

\usepackage{latexsym}
\usepackage{amsmath}
\usepackage{amsfonts}
\usepackage{amssymb}
\usepackage{enumerate}
\usepackage{relsize}
\usepackage{color}
\usepackage{graphicx}
\usepackage{float}
\usepackage[margin=0.8in]{geometry}
\usepackage[dvipsnames]{xcolor}
\usepackage{outlines}
\usepackage{soul}
\usepackage{comment}
\usepackage{BeamerColor}
\usepackage{tabularx}
\usepackage{caption}
\usepackage{blindtext}
\usepackage[all]{xy}
\usepackage{tikz-cd}
\usetikzlibrary{matrix,arrows} 
\usepackage{mathtools}
\usepackage{adjustbox} 
\usepackage{mwe} 
\usepackage{mathtools} 
\usepackage{accents} 
\usepackage{tikz-cd}
\usepackage{setspace}

\usepackage[colorlinks = true,
            linkcolor = RoyalBlue,
            urlcolor  = blue,
            citecolor = blue,
            anchorcolor = blue]{hyperref}

\newtheorem{theorem}[equation]{Theorem}

\newtheorem{lemma}[equation]{Lemma}

\newtheorem{proposition}[equation]{Proposition}
\newtheorem{corollary}[equation]{Corollary}

\newtheorem{definition-lemma}[equation]{Definition-Lemma}
\theoremstyle{definition}
\newtheorem{definition}[equation]{Definition}

\theoremstyle{remark}
\newtheorem{remark}[equation]{Remark}

\numberwithin{equation}{section}
\numberwithin{figure}{section}

\newcommand {\Adm} {\operatorname{\text{Adm}}}

\newcommand {\Aut} {\operatorname{\text{Aut}}}

\newcommand {\Cl} {\operatorname{\text{Cl}}}
\newcommand {\Conv} {\operatorname{\text{Conv}}}
\newcommand {\Curv} {\operatorname{\text{\ovA{Curv}}}}
\newcommand {\oCurv} {\operatorname{\text{Curv}}}

\newcommand {\Eff}  {\operatorname{Eff}}
\newcommand {\Fix}  {\operatorname{Fix}}

\newcommand {\Int}  {\operatorname{Int}}

\newcommand {\Nef} {\operatorname{\text{Nef} \;}}
\newcommand {\Nefe} {\operatorname{\Nef^{e}}}
\newcommand {\Pic}  {\operatorname{Pic}}

\newcommand {\relInt}  {\operatorname{relInt}}

\newcommand {\Supp} {\operatorname{\text{Supp}}}

\newcommand{\ba}{\backslash}

\newcommand{\vsimeq}{\mathrel{\rotatebox{90}{$\simeq$}}}

\makeatletter
\newcommand*{\ov}[1]{%
  $\m@th\overline{\mbox{#1}}$%
}
\newcommand*{\ovA}[1]{%
  $\m@th\overline{\mbox{#1}\raisebox{3mm}{}}$%
}
\newcommand*{\ovB}[1]{%
  $\m@th\overline{\mbox{#1\rule{0pt}{3mm}}}$%
}
\newcommand*{\ovC}[1]{%
  $\m@th\overline{\mbox{#1\strut}}$%
}
\newcommand*{\ovD}[1]{%
  $\m@th\overline{\mbox{#1\vphantom{\"A}}}$%
}
\newcommand*{\ovE}[1]{%
  $\m@th\overline{\raisebox{0pt}[1.2\height]{#1}}$%
}
\newcommand*{\ovF}[1]{%
  $\m@th\overline{\raisebox{0pt}[\dimexpr\height+1mm\relax]{#1}}$%
}
\newcommand*{\ovG}[1]{%
  $\m@th\overline{\raisebox{0pt}[\dimexpr\height+1mm\relax]{#1\vphantom{A}}}$%
}
\makeatother

\begin{document}

\title{A cone conjecture for log Calabi-Yau surfaces}
\author{Jennifer Li}

\maketitle

%
%

\begin{abstract}
We consider log Calabi-Yau surfaces $(Y, D)$ with singular boundary. In each deformation type, there is a distinguished surface $(Y_e,D_e)$ such that the mixed Hodge structure on $H_2(Y \setminus D)$ is split. We prove that (1) the action of the automorphism group of $(Y_e,D_e)$ on its nef effective cone admits a rational polyhedral fundamental domain; and (2) the action of the monodromy group on the nef effective cone of a very general surface in the deformation type admits a rational polyhedral fundamental domain. These statements can be viewed as versions of the Morrison cone conjecture for log Calabi--Yau surfaces. In addition, if the number of components of $D$ is $\le 6$, we show that the nef cone of $Y_e$ is rational polyhedral and describe it explicitly. This provides infinite series of new examples of Mori Dream Spaces.
\end{abstract}

%
%

\section{Introduction} Given a smooth projective variety $Y$ over $\mathbb{C}$, the closed cone of curves of $Y$ is the closure of the set of all nonnegative linear combinations of classes of irreducible curves in $H_{2}(Y, \mathbb{R})$. The cone of curves of any Fano variety is rational polyhedral, meaning it has finitely many rational generators (see Theorem 1.24 on p.22 of \cite{KM98}). But this is not true in general for Calabi-Yau varieties - if $Y$ is Calabi-Yau, the cone of curves of $Y$ could be round, for example. The nef cone is the dual of the cone of curves.

The Morrison cone conjecture states that if $Y$ is a Calabi-Yau variety, then there exists a rational polyhedral cone which is a fundamental domain for the action of the automorphism group of $Y$ on the nef cone. This can be pictured in dimension two using hyperbolic geometry (see e.g., \cite{T11}).

The conjecture is known to be true in dimension two, but for higher dimensions, it is an open question. In \cite{T10}, Totaro has shown that a generalization of this conjecture is true in dimension two: if $(Y, \Delta)$ is a klt Calabi-Yau pair, then the automorphism group of $Y$ acts on the nef cone with a rational polyhedral fundamental domain.

We study a cone conjecture for log Calabi-Yau surfaces that is similar to, but different from, the conjecture proved by Totaro. Let $Y$ be a smooth projective surface and $D$ a reduced normal crossing divisor on $Y$ such that $K_{Y} + D = 0$. We call $(Y, D)$ a \emph{log Calabi-Yau surface}. Additionally, we require $D$ to be singular, and write $D = D_{1} + \dots + D_{n}$ for the irreducible components of $D$. By the Gross-Hacking-Keel Torelli theorem for log Calabi-Yau surfaces (\cite{GHK15b}, Theorem 1.8), in each deformation type of log Calabi-Yau surfaces there exists a unique pair $(Y, D) = (Y_{e}, D_{e})$ such that the mixed Hodge structure on $Y \setminus D$ is split. The main result of this paper is the proof of the following statement (see Theorem \ref{coneConjectureYgen} and Theorem \ref{coneConjectureYprimegen}):

\begin{theorem}
\label{thm:coneConjecture}
Consider a deformation type of log Calabi-Yau surfaces $(Y, D)$ with singular boundary.
\begin{enumerate}
\item Let $(Y_{e}, D_{e})$ be the unique surface in this deformation type with split mixed Hodge structure. Let $K$ be the kernel of the action of the automorphism group of the pair on $H^{2}(Y, \mathbb{Z})$. Then $\Aut(Y_{e}, D_{e})/K$ acts on the nef effective cone $\Nefe(Y_{e})$ with a rational polyhedral fundamental domain.
\item Let $(Y_{gen}, D_{gen})$ be a very general surface in this deformation type. Then the monodromy group $\Adm$ acts on the nef effective cone $\Nefe(Y_{gen})$ with a rational polyhedral fundamental domain.
\end{enumerate}
\end{theorem}

The Morrison cone conjecture, stated in 1993, is originally inspired by mirror symmetry. The log Calabi-Yau surface version of this conjecture is also related to mirror symmetry through the deformation theory of cusp singularities of surfaces. 

Given a log Calabi-Yau surface $(Y, D)$ such that the intersection matrix $(D_{i} \cdot D_{j})$ is negative definite, we may contract the boundary $D$ to obtain a normal surface $Y$ with a cusp singularity $p \in Y^{\prime}$ (see Grauert \cite{G62} and Definition \ref{def:cuspSingularity}). Cusp singularities come in dual pairs such that the links are diffeomorphic but have opposite orientations. If $(Y^{\prime}, p)$ is obtained by contracting the boundary of a log Calabi-Yau surface $(Y, D)$ to a cusp singularity $p \in Y^{\prime}$, then, conjecturally, $(Y, D)$ corresponds to an irreducible component of the deformation space of the dual cusp (\cite{GHK15a}, \cite{E15}, \cite{EF16}). This is expected as a consequence of mirror symmetry: $Y \ba D$ is mirror to the Milnor fiber of the corresponding smoothing of the dual cusp (\cite{Ke15}, \cite{HKe21}). Again conjecturally, the component of the deformation space of the dual cusp can be described in terms of the action of the monodromy group $\Adm$ on $\Nef(Y^{\prime})$, by a construction of Looijenga (\cite{L03}, $\S 4$). However, to use this construction, the group $\Adm$ must act with a rational polyhedral fundamental domain on the effective nef cone of $Y^{\prime}$, and this is the original motivation for our conjecture, cf. \cite{M93}.

This paper is organized as follows. Sections \ref{Background}-\ref{ProofOfTheConjecture} contain the proof of the main theorem. In Section \ref{NewExamplesMDS}, we give explicit descriptions of certain cones of curves, which provide infinite series of new examples of Mori Dream Spaces. We end the paper with Section \ref{Motivations}, where we explain in some more detail the motivations of our project.

\noindent {\bf Acknowledgements.} This paper is based on results from my PhD thesis. I am grateful to my advisor, Paul Hacking, for his guidance on this project. I thank Ben Heidenreich, Eyal Markman, Jenia Tevelev, and Burt Totaro for their insightful comments and suggestions. I also thank Alejandro Morales, John Christian Ottem, Luca Schaffler, Angelica Simonetti, and Sebasti\'{a}n Torres for helpful discussions. My research was partially supported by NSF grant DMS-1901970.

%
%

\section{BACKGROUND}
\label{Background}

Let $Y$ be a smooth projective variety. We define $N^{1}(Y)$ to be the space of divisors with real coefficients modulo numerical equivalence, and the space $N_{1}(Y)$ to be the space of 1-cycles with real coefficients modulo numerical equivalence. Because $Y$ is a rational surface in our setting,
\begin{align}
\label{eqn:PicYClY}
N^{1}(Y) &= H^{2}(Y, \mathbb{R}) = \Pic(Y) \otimes \mathbb{R}, \\
N_{1}(Y) &= H_{2}(Y, \mathbb{R}) = \Cl(Y) \otimes \mathbb{R},
\end{align}
and $N^{1}(Y) = N_{1}(Y)$. We define the {\it nef cone of $Y$} to be
\begin{center}
$\Nef(Y) = \{ L \in N^{1}(Y) \; \vert \; L \cdot C \geq 0 \text{ for all irreducible curves } C \subset Y \}$.
\end{center}
The {\it effective cone of $Y$} is
\begin{center}
$\Eff(Y) = \bigl\{ \displaystyle{\sum} a_{i} [D_{i}] \in N^{1}(Y) \; \vert \; a_{i} \in \mathbb{R}_{\geq 0} \text{ and } D_{i} \subset Y \text{ are codimension one subvarieties} \bigr\}$.
\end{center}
Following Kawamata in \cite{K97}, we define the {\it nef effective cone of $Y$} to be
\begin{center}
$\Nefe(Y) = \Nef(Y) \; \bigcap \; \Eff(Y)$.
\end{center}
\noindent We denote the {\it convex hull} of the set $S$ by $\Conv(S)$, where $S$ is subset of a real vector space. If $Y$ is a surface, then 
\begin{equation}
\label{nefeConvexHullNefYPicY}
\Nefe(Y) = \Conv \, \bigl( \{[L] \in N^{1}(Y) \; \vert \; L \in \Pic(Y) \text{ is nef and } h^{0}(L) \neq 0\} \bigr).
\end{equation}

\begin{definition}
\label{def:OpenConeOfCurves}
The {\it cone of curves of $Y$} is defined as follows:
\begin{center}
$\oCurv(Y) = \bigl\{ \displaystyle{\sum} a_{i} [C_{i}] \in N_{1}(Y) \; \vert \; a_{i} \in \mathbb{R}_{\geq 0} \text{ and each } C_{i} \subset Y \text{ an irreducible curve} \bigr\}.$
\end{center}
We write $\Curv(Y)$ to mean the closure of the cone of curves. 
\end{definition}

\begin{definition}
\label{def:RPcone}
Let $L$ be a finitely generated free Abelian group, ie., $L \simeq \mathbb{Z}^{\rho}$ for some $\rho \geq 0$. A cone $C \subset L \otimes_{\mathbb{Z}} \mathbb{R} \simeq \mathbb{R}^{\rho}$ is said to be {\it rational polyhedral} if
\begin{center}
$C = \langle v_{1}, \dots, v_{r} \rangle_{\mathbb{R} \geq 0} = \{a_{1}v_{1} + \dots + a_{r}v_{r} \; \vert \; a_{i} \in \mathbb{R}_{\geq 0} \}$,
\end{center}
for some $v_{1}, \dots, v_{r} \in L$. That is, the cone $C$ is generated by finitely many integral vectors $v_{1}, \dots, v_{r} \in L$.
\end{definition}

\begin{definition}
\label{def:logCYpair}
A {\it log Calabi-Yau surface} is a pair $(Y, D)$ where $Y$ is a smooth complex projective surface and $D \subset Y$ is a reduced normal crossing divisor such that $K_{Y} + D = 0$. We say that $(Y, D)$ has \emph{maximal boundary} if $D$ is singular. We write $D = D_{1} + \dots + D_{n}$, where $n$ is the number of irreducible components or the {\it length} of $D$.
\end{definition}

\noindent In this paper, we always assume that $(Y, D)$ has maximal boundary. If $(Y, D)$ is a log Calabi-Yau surface with maximal boundary, then $Y$ is a rational surface (\cite{GHK15b}, p.2).

\begin{remark}
\label{rem:boundary}
The boundary $D$ is either a rational curve of arithmetic genus one with a single node (i.e., a copy of $\mathbb{P}^{1}$ with two points identified to form a node), or it is a cycle of smooth rational curves (i.e., a cycle of $n$ copies of $\mathbb{P}^{1}$). This follows from the adjunction formula. We fix a cyclic ordering $D = D_{1} + \dots + D_{n}$ of the components of $D$ and a compatible orientation (an isomorphism $H_{1}(D, \mathbb{Z}) \simeq \mathbb{Z})$. This orientation is uniquely determined by the cyclic ordering for $n > 2$.
\end{remark}

\begin{definition} 
\label{def:genericPair}
We say that a log Calabi-Yau surface $(Y, D)$ is {\it generic} if there are no $(-2)$-curves $C$ contained in $Y \setminus D$. We sometimes write $(Y_{gen}, D_{gen})$ to denote one such log Calabi-Yau surface in a given deformation type.
\end{definition}

\begin{definition}
\label{def:deformationEquivalent}
Two log Calabi-Yau surfaces $(Y^{1}, D^{1})$ and $(Y^{2}, D^{2})$ are said to be {\it deformation equivalent} if there exists a flat family $(\mathcal{Y}, \mathcal{D}) = (\mathcal{Y}, \mathcal{D}_{1} + \cdots + \mathcal{D}_{n})$ of log Calabi-Yau surfaces over a connected base $S$ such that there are points $p, q \in S$ with fibers $f^{-1}(p) = (Y^{1}, D^{1})$ and $f^{-1}(q) = (Y^{2}, D^{2})$. In this case, we say that $(Y^{1}, D^{1})$ and $(Y^{2}, D^{2})$ are of the same {\it deformation type}.
\end{definition}

By the Torelli Theorem in \cite{GHK15b}, given a log Calabi-Yau surface $(Y, D)$, the moduli space $\mathcal{M}$ of log Calabi-Yau surfaces that are deformation equivalent to $(Y, D)$ can be described explicitly and the locus of generic surfaces is the complement of a countable union of divisors in $\mathcal{M}$ (see \cite{GHK15b}, Section 6). For any two generic surfaces of the same deformation type, the nef cones of the two surfaces are the same. This cone for $Y_{gen}$ is described after the following definition:

\begin{definition}
\label{def:interior-1curve}
For a log Calabi-Yau surface $(Y, D)$, an {\it interior $(-1)$-curve} is a smooth rational curve of self-intersection $-1$ that is not contained in the boundary $D$. By the adjunction formula, such a curve must intersect the boundary transversely at a single point.
\end{definition}

\begin{proposition} ( \cite{GHK15b}, Lemma 2.15)
\begin{align*}
\Nef(Y_{gen}) = \{ L \in \Pic(Y) \otimes_{\mathbb{Z}} \mathbb{R} \; \vert \; &L^{2} \geq 0 \text{ and } L \cdot D_{i} \geq 0 \; \text{ for all } \; i \text{ and } \\
&L \cdot C \geq 0 \; \text{ for any interior (-1)-curve } C\}.
\end{align*}
\end{proposition}

\begin{lemma}
\label{YgenLnefImpliesEffective}
Let $(Y, D)$ be a log Calabi-Yau surface. If $L \in \Pic(Y)$ is nef, then $L$ is effective.
\end{lemma}

\begin{proof} 
Let $L \in \Pic(Y)$ be nef. By Riemann-Roch, we have
\begin{align*}
\chi(L) &= \chi(\mathcal{O}_{Y}) + \displaystyle{\frac{1}{2} L(L - K_{Y})} \\
&= 1 + \displaystyle{\frac{1}{2}(L^{2} + L \cdot D)} \\
&\geq 1,
\end{align*}
since $L$ being nef and $D$ being effective give $L \cdot D \geq 0$ and $L$ nef gives $L^{2} \geq 0$. On the other hand, we have
\begin{align*}
\chi(L) &= h^{0}(L) - h^{1}(L) + h^{2}(L) \\
&\leq h^{0}(L) + h^{2}(L)
\end{align*}

Next we show that $h^{2}(L) = 0$. By Serre Duality, we have $h^{2}(L) = h^{0}(K_{Y} - L) = h^{0}(-D-L)$. If $H$ is ample and $L$ is nef and $D$ is effective, then we have $H \cdot D > 0$ and $H \cdot L \geq 0$. Then $H \cdot (-D - L) < 0$, so $h^{0}(-D-L) = 0$. Thus $h^{0}(L) \geq \chi(L) \geq 1$, and therefore $L$ is linearly equivalent to an effective divisor.
\end{proof}

\begin{definition}
\label{def:cuspSingularity}
A {\it cusp singularity} is a surface singularity whose minimal resolution is a cycle of smooth rational curves that meet transversally. That is, the exceptional locus of the minimal resolution of a cusp singularity is a union of copies of $\mathbb{P}^{1}$ with nodal singularities such that the dual graph is a cycle.
\end{definition}

Given a log Calabi-Yau surface $(Y, D)$ with $D$ having a negative definite intersection matrix $(D_{i} \cdot D_{j})$, it is possible to contract $D$ to a cusp singularity $p$ (by a theorem of Grauert on the contractibility of a negative definite configuration of curves on a smooth complex surface in the analytic category, \cite{G62}). Let $f: Y \rightarrow Y^{\prime}$ be the morphism contracting $D$ to a point. Then we have the induced isomorphism 
\begin{center}
$Y \setminus D \cong Y^{\prime} \setminus \{p\}$,
\end{center}
and $f^{-1}(p) = D$. In addition, the surface $Y^{\prime}$ is normal and compact (for the usual Euclidean topology). We note that although $Y$ is a projective variety, the new surface $Y^{\prime}$ is in general no longer a projective variety, but a normal, analytic space. We make the following definitions.

\begin{definition} We define the nef effective cone of $Y^{\prime}$ in the following way:
\begin{center}
$\Nefe(Y^{\prime}) := \Nefe(Y) \cap \langle D_{1}, \dots, D_{n} \rangle^{\bot}$.
\end{center}
\end{definition}

\begin{remark}
\label{rem:equivalentNefYprimeDef}
Equivalently $\Nefe(Y^{\prime}) = \Nef(Y^{\prime}) \cap \Eff(Y^{\prime})$, where
\begin{center}
$\Nef(Y^{\prime}) := \{ L \in Cl(Y^{\prime}) \otimes \mathbb{R} \; \vert \; L \cdot C \geq 0 \text{ for all curves } C \subset Y^{\prime}\}$,
\end{center}
and we use Mumford's intersection product on a normal surface $Y^{\prime}$ (\cite{M61}, p.17). Note that $Y^{\prime}$ is not $\mathbb{Q}$-factorial in general (i.e., there may exist divisors which are not $\mathbb{Q}$-Cartier).
\end{remark}

\begin{definition}
\label{def:autYD}
An {\it isomorphism of log Calabi-Yau surfaces} $(Y^{1}, D^{1})$ and $(Y^{2}, D^{2})$ is an isomorphism $\theta: Y^{1} \rightarrow Y^{2}$, with the property that $\theta(D^{1}_{i}) = \theta(D^{2}_{i})$ for each boundary component $D^{k}_{i}$ of $D^{k}$ for $k = 1, 2$, and $\theta$ respects the orientations of $D^{1}$ and $D^{2}$ (automatic for $n \geq 3$).
\end{definition}

\begin{definition}
\label{def:admissibleGroup}
Given any log Calabi-Yau surface $(Y, D)$, the {\it admissible group} of $Y$ is defined as follows:
\begin{align*}
\Adm = \{ \theta \in \Aut(\Pic(Y)) \; \vert \; &\theta([D_{i}]) = [D_{i}] \text{ for all } i = 1, \dots, n \text{ and } \\
 &\theta(\Nef(Y_{gen})) = \Nef(Y_{gen}) \}.
\end{align*}
\end{definition}

\begin{remark}
$\Adm$ is identified with the monodromy group for $(Y, D)$ (\cite{GHK15b}, Theorem 5.15).
\end{remark}

\begin{definition}
\label{def:fundamentalDomainOfGroupAction}
Let $\Gamma$ be a group and $X$ a topological space. Suppose that $\Gamma$ acts on $X$ by homeomorphisms. We say that a closed subset $D \subset X$ is a {\it fundamental domain} for the action of $\Gamma$ on $X$ if the following are true:
\begin{enumerate}
\item for all $x \in X$, there exists $d \in D$ and $\gamma \in \Gamma$ such that $\gamma(d) = x$; and
\item for all $\gamma_{1}, \gamma_{2} \in \Gamma$ such that $\gamma_{1} \neq \gamma_{2}$, the intersection $\gamma_{1}D \cap \gamma_{2}D$ has empty interior.
\end{enumerate}
\end{definition}

\begin{definition}
\label{def:periodPoint}
Given a log Calabi-Yau surface $(Y, D)$, the {\it period point} is defined to be the homomorphism $\phi: \langle D_{1}, \dots, D_{n} \rangle^{\bot} \rightarrow \mathbb{C}^{\ast}$, where a line bundle $L \in \langle D_{1}, \dots, D_{n} \rangle^{\bot}$ is sent to $\theta([L \vert_{D}]) \in \mathbb{C}^{\ast}$, where $\theta: \Pic^{0}(D) \xrightarrow{\sim} \mathbb{C}^{\ast}$ is the isomorphism determined by the given orientation of $D$ (as explained in \cite{GHK15b}, Lemma 2.1). Here $\Pic^{0}(D)$ is the kernel of the map $c_{1}: \Pic(D) \rightarrow H^{2}(D, \mathbb{Z}) = \mathbb{Z}^{n}$, given by $L \mapsto (\deg L \vert_{D_{i}})_{i=1}^{n}$.
\end{definition}

By \cite{F15}, Proposition 3.12, the homomorphism $\phi$ is the extension class of the mixed Hodge structure on $H_{2}(U, \mathbb{C})$, where we take $U = Y \setminus D$. There is an exact sequence (\cite{L81}, Chapter I, Section 5.1):
\begin{center}
$0 \rightarrow \mathbb{Z} \rightarrow H_{2}(U) \rightarrow \langle D_{1}, \dots, D_{n} \rangle^{\bot} \rightarrow 0.$
\end{center}

There exists a unique log Calabi-Yau surface in each deformation type such that $\phi(\alpha) = 1$ for all $\alpha \in \langle D_{1}, \dots, D_{n} \rangle^{\bot}$, i.e., the mixed Hodge structure on $H_{2}(U)$ is split. This follows from the Torelli theorem (\cite{GHK15b}, \S 5; and \cite{F15}, Corollary 9.7). We denote this log Calabi-Yau surface by $(Y_{e}, D_{e})$.

\begin{definition}
\label{def:rootSystem}
Given a log Calabi-Yau surface $(Y, D)$, the associated {\it root system} is the subset of $\Pic(Y)$ defined by:
\begin{center}
$\Phi = \{ \alpha \in \langle D_{1}, \dots, D_{n} \rangle^{\bot} \; \vert \; \alpha^{\bot} \cap \Int(\Nef(Y_{gen})) \neq \emptyset \text{ and } \alpha^{2} = -2 \}$
\end{center}
\end{definition}

\begin{definition}
\label{def:WeylGroup1}
We define the {\it Weyl group} of the root system $\Phi \subset \Pic(Y)$ as follows:
\begin{center}
$W = \langle s_{\alpha} \; \vert \; \alpha \in \Phi \rangle \subset \Aut(\Pic(Y), \cdot)$,
\end{center}
where the generators $s_{\alpha}(\beta) = \beta + (\alpha \cdot \beta) \alpha$ are the reflections in the hyperplanes $\alpha^{\bot}$ for $\alpha \in \Phi$.
\end{definition}

\begin{definition}
\label{def:simpleRoots}
Given $(Y_{e}, D_{e})$, we define the {\it simple roots} as the set:
\begin{center}
$\Delta = \{ [C] \; \vert \; C \subset Y_{e} \setminus D_{e} \text{ is a $(-2)$-curve} \}$.
\end{center}
\end{definition}

\begin{proposition}
\label{prop:WeylGroupPresentation} 
( \cite{GHK15b}, Proposition 3.4) The set $\Delta$ is contained in $\Phi$ and the Weyl group $W$ is generated by the reflections $s_{\delta}$ for $\delta \in \Delta$, i.e., 
\begin{center}
$W = \langle s_{\delta} \; \vert \; \delta \subset \Delta \rangle$.
\end{center}
\end{proposition}

\noindent By \cite{GHK15b}, Lemma 2.15,
\begin{center}
$\Nef(Y_{e}) = \Nef(Y_{gen}) \; \bigcap \; (\delta \geq 0 \text{ for all } \delta \in \Delta)$.
\end{center}

\begin{remark}
The Weyl group is a normal subgroup of $\Adm$. (Proof: it follows from Definitions \ref{def:admissibleGroup} and \ref{def:rootSystem} that $\Adm$ preserves $\Phi$. If $g \in \Adm$ and $\alpha \in \Phi$, then $g s_{\alpha} g^{-1} = s_{g(\alpha)}$, which implies that $W \lhd \Adm$).
\end{remark}

By \cite{GHK15b}, Theorem 3.2, the group $W$ acts on $\Nefe(Y_{gen})$ with fundamental domain $\Nefe(Y_{e})$. This is called the {\it fundamental chamber} in $\Nef(Y_{gen})$. By \cite{GHK15b}, Theorem 5.1, there is an exact sequence
\begin{center}
$1 \rightarrow K \rightarrow \Aut(Y_{e}, D_{e}) \rightarrow \Adm / W \rightarrow 1$,
\end{center}
where $K$ is the kernel of the action of $\Aut(Y_{e}, D_{e})$ on $\Pic(Y)$.

%
%

\section{TOOLS}
\label{Tools}
Here we include some main results that we use in the proof of our results.
\label{Tools}

\begin{theorem} 
\label{thm:ghkTorelli}
( \cite{GHK15b}, Theorem 1.8, The global Torelli theorem for log Calabi-Yau surfaces). Suppose that $(Y^{1}, D^{1})$ and $(Y^{2}, D^{2})$ are log Calabi-Yau surfaces. Consider the following three statements:
\begin{enumerate}
\item $\theta: \Pic(Y^{1}) \rightarrow \Pic(Y^{2})$ is an isometry such that $\theta([D_{i}^{1}]) = [D_{i}^{2}]$ for $i = 1, \dots, n$.
\item $\theta(L)$ is ample for some ample $L$ on $Y^{1}$.
\item $\phi_{Y^{2}} \circ \theta = \phi_{Y^{1}}$, where $\phi_{Y}: \langle D_{1}, \dots, D_{n} \rangle^{\bot} \rightarrow \mathbb{C}^{\ast}$ is the period point of $Y$.
\end{enumerate}
$(1), (2)$ and $(3)$ hold if and only if $\theta = f^{\ast}$ for some isomorphism $f: (Y^{2}, D^{2}) \rightarrow (Y^{1}, D^{1})$.
\end{theorem}

\begin{proposition} (\;\cite{EF16}, Proposition 1.5).
\label{prop:EngelFriedman} Let $(Y_{gen}, D_{gen})$ be a generic log Calabi-Yau surface, where $D_{gen}$ has at least three boundary components. If $B$ is a divisor on $Y_{gen}$ with nonnegative integer coefficients, then $B$ is linearly equivalent to a divisor of the form
\begin{center}
$\displaystyle{\sum a_{i}D_{i} + \sum b_{j}E_{j}}$,
\end{center}
where the $E_{j}$'s are disjoint interior $(-1)$-curves and $a_{i}, b_{j}$ are nonnegative integers.
\end{proposition}

\begin{remark}
\label{rem:EF-realCoefficients}
Although the Engel-Friedman Proposition \ref{prop:EngelFriedman} is stated for $B$ with nonnegative {\it integer} coefficients, the statement also holds for $B$ with nonnegative {\it real} coefficients. There is a sketch of the proof in the paper (\cite{EF16}, p.55), which uses a continuity argument and the assertion that the collection of subsets
\begin{center}
$\bigg\{ \displaystyle{\sum a_{j}D_{j} + \sum b_{i}E_{i} \; \vert \; a_{j}, b_{i} \in \mathbb{R}_{\geq 0}} \bigg\}$,
\end{center}
where the $E_{i}$'s are disjoint interior $(-1)$-curves, is locally finite in $\Nefe(Y_{gen})$ in a sense that is made precise below. Since this is important for our results, we give a complete proof (Proposition \ref{prop:EF-realCoefficients}).
\end{remark}

Friedman showed in \cite{F15} that $\Adm$ acts with finitely many orbits on the set of faces of $\Nefe(Y_{gen})$ corresponding to interior $(-1)$-curves. This is stated in Theorem \ref{thm:FriedmanFiniteOrbits}.

\begin{theorem}(Friedman \cite{F15}, Theorem 9.8)
\label{thm:FriedmanFiniteOrbits}
Let $(Y, D)$ be a generic log Calabi-Yau surface. Let $\mathcal{E}(Y, D)$ be the set of all interior $(-1)$-curves of $Y$. Then the admissible group $\Adm$ acts on $\mathcal{E}(Y, D)$ and there are finitely many $\Adm$-orbits for this action.
\end{theorem}

The following Corollary \ref{cor:FriedmanMinus1} by Friedman is similar to the statement above. Specifically, it is a statement about the action of $\Adm$ on the set of collections of disjoint interior $(-1)$-curves.

\begin{corollary}(Friedman \cite{F15}, Corollary 9.10)
\label{cor:FriedmanMinus1}
Given a generic log Calabi-Yau surface $(Y, D)$, let $\mathcal{E}_{k}(Y, D)$ be the set of collections $\{E_{1}, \dots, E_{k}\}$, for any $k \in \mathbb{N}$, where the curves $E_{i}$ are disjoint, interior $(-1)$-curves. Then the admissible group $\Adm$ acts on $\mathcal{E}_{k}(Y, D)$ and the number of $\Adm$ orbits for this action is finite.
\end{corollary}

\begin{theorem}(Looijenga \cite{L14}, Proposition-Definition 4.1; and Application 4.14; and Proposition 4.7)
\label{prop-def:Looijenga}
Let $\Gamma$ be a group and $L$ be a lattice, i.e., a finitely generated free abelian group, and let $C \subset L \otimes_{\mathbb{Z}} \mathbb{R}$ be an open nondegenerate convex cone. Define
\begin{center}
$C_{+} := \Conv (\bar{C} \cap L)$.
\end{center}
Assume that $\Gamma$ acts on $L$ faithfully, preserving the cone $C$. If there exists a polyhedral cone $\Pi \subset C_{+}$ such that $\Gamma \cdot \Pi = C_{+}$, then there exists a rational polyhedral fundamental domain for the action of $\Gamma$ on $C_{+}$. Moreover, in this case, the group $N_{\Gamma} F / Z_{\Gamma} F$ acts on any face $F$ of $C_{+}$ with a rational polyhedral fundamental domain.
\end{theorem}

\begin{remark}
In Theorem \ref{prop-def:Looijenga}, following the notation of Looijenga, we use $N_{\Gamma} F$ to mean the normalizer of $F$ in $\Gamma$ and $Z_{\Gamma} F$ to mean the centralizer (i.e., elements of $\Gamma$ that fix $F$ pointwise). The last statement is a special case of Proposition 4.7 in \cite{L14}.
\end{remark}

\begin{proposition}
\label{prop:CprimeCover}
Let $(Y_{gen}, D_{gen})$ be a generic log Calabi-Yau surface. For a collection $\{E_{1}, \dots, E_{k}\}$ of disjoint $(-1)$-curves, define 
\begin{align*}
C^{\prime}(E_{1}, \dots, E_{k}) := \langle D_{1}, \dots, D_{n}, E_{1}, \dots, E_{k} \rangle_{\mathbb{R}_{\geq 0}} \cap \Nefe(Y_{gen}).
\end{align*}
Then
\begin{enumerate}
\item $C^{\prime}(E_{1}, \dots, E_{k})$ is a rational polyhedral cone; and
\item If $D_{gen}$ consists of at least three components, then the set of cones $C^{\prime}(E_{1}, \dots, E_{k})$ covers $\Nefe(Y^{\prime}_{gen})$.
\end{enumerate}
\end{proposition}

\begin{proof}
Let $C$ be the cone defined by
\begin{center}
$C = C(E_{1}, \dots, E_{k}) := \langle D_{1}, \dots, D_{n}, E_{1}, \dots, E_{k} \rangle_{\mathbb{R}_{\geq 0}}$,
\end{center}
so that $C^{\prime}(E_{1}, \dots, E_{k})$ can be expressed as $C \cap \Nefe(Y_{gen})$. The cone $C^{\prime}(E_{1}, \dots, E_{k})$ is rational polyhedral because $C$ is rational polyhedral by definition, and the intersection with the nef cone is given by finitely many inequalities $L \cdot D_{i} \geq 0$ and $L \cdot E_{j} \geq 0$ for all $1 \leq i \leq n$ and all $1 \leq j \leq k$. This shows that $C^{\prime}(E_{1}, \dots, E_{k})$ is rational polyhedral. Now assume that $D_{gen}$ has at least three components. Then, by Proposition \ref{prop:EF-realCoefficients},
\begin{center}
$\Nefe(Y_{gen}) = \bigcup \; C^{\prime}(E_{1}, \dots, E_{k})$,
\end{center}
where the union is over the set $\displaystyle{\bigcup_{k} \mathcal{E}_{k}(Y, D)}$ of collections $\{E_{1}, \dots, E_{k}\}$ of disjoint interior $(-1)$-curves.
\end{proof}

\begin{theorem}
\label{thm:TheSiegelProperty}
(The Siegel Property, as stated in \cite{L14}, Theorem 3.8.)
Let $L$ be a lattice and $V = L \otimes \mathbb{R}$. Let $C \subset V$ be an open convex nondegenerate cone. We denote the convex hull $\Conv(\bar{C} \cap L)$ by $C_{+}$. Let $\Gamma$ be a subgroup of $GL(V)$ such that $\Gamma$ leaves the cone $C$ and the lattice $L$ invariant. Then $\Gamma$ has the Siegel Property in $C_{+}$, that is, if $\Pi_{1}$ and $\Pi_{2}$ are polyhedral cones in $C_{+}$, then the collection $\{\gamma \Pi_{1} \bigcap \Pi_{2}\}_{\gamma \in \Gamma}$ is finite. 
\end{theorem}

Next, we give the precise statement of Engel-Friedman's Proposition \ref{prop:EngelFriedman} for real coefficients, followed by a careful proof, which uses the Siegel Property \ref{thm:TheSiegelProperty} and the result \ref{cor:FriedmanMinus1} of Friedman.

\begin{proposition}
\label{prop:EF-realCoefficients}
(Proposition \ref{prop:EngelFriedman} for $B$ with real coefficients)
Let $(Y_{gen}, D_{gen})$ be a generic log Calabi-Yau surface, where $D_{gen}$ has at least three boundary components. If $B$ is a divisor on $Y_{gen}$ with nonnegative real coefficients, then $B$ is $\mathbb{R}$-linearly equivalent to a divisor of the form
\begin{center}
$\displaystyle{\sum a_{i}D_{i} + \sum b_{j}E_{j}}$,
\end{center}
where the $E_{j}$'s are disjoint interior $(-1)$-curves and $a_{i}, b_{j}$ are nonnegative real numbers.
\end{proposition}

\begin{proof}
We use the same notation introduced in Proposition \ref{prop:CprimeCover} above, that is, for a log Calabi-Yau surface $(Y_{gen}, D_{gen})$ where $D_{gen}$ is of length at least three, we let
\begin{center}
$C := \langle D_{1}, \dots, D_{n}, E_{1}, \dots, E_{k} \rangle_{\mathbb{R} \geq 0}$.
\end{center}
We want to show that
\begin{equation}
\label{eqn:EFCor-equality}
\oCurv(Y_{gen}) = \bigcup \; C(E_{1}, \dots, E_{k}),
\end{equation}
where $\oCurv(Y) := \{ \sum a_{i} [C_{j}] \; \vert \; a_{i} \in \mathbb{R}_{\geq 0} \text{ and } C_{i} \subset Y \text{ are irreducible curves}\}$ (Note, because $\dim(Y_{gen}) = 2$, the cones $\Eff(Y_{gen})$ and $\oCurv(Y_{gen})$ coincide). Clearly,
\begin{center}
$\displaystyle{\bigcup} \; C(E_{1}, \dots, E_{k}) \subseteq \oCurv(Y_{gen})$,
\end{center}
where the union is taken over all $C(E_{1}, \dots, E_{k})$ where $\{E_{1}, \dots, E_{k}\}$ are collections of disjoint interior $(-1)$-curves.

Let $x \in \oCurv(Y_{gen})$ be an arbitrary point. A convex cone is the disjoint union of the relative interiors of its faces. This follows from the supporting hyperplane theorem (\cite{S11}, Proposition 8.5).

So $x \in \relInt(F)$, where $F$ is some face of $\oCurv(Y_{gen})$ (possibly $F = \oCurv(Y_{gen})$). Since $\oCurv(Y_{gen})$ is generated by rational points, the same is true for any face of $\oCurv(Y_{gen})$. In particular, the face $F$ is the convex hull of its rational points, so the rational points are dense in $F$. Thus we may choose a sequence of points $x_{n} \in F \; \bigcap \; (\Pic(Y) \otimes \mathbb{Q})$ that converge to $x$ as $n$ approaches infinity.  The original Engel-Friedman statement (see Proposition \ref{prop:EngelFriedman}) was stated for integer coefficients, but this implies that the statement for rational coefficeints is also true. So for every $n$, the point $x_{n}$ belongs to some cone $C(E_{1}, \dots, E_{k})$, as defined above.

Since $x \in \relInt(F)$, i.e., the interior of $F$ regarded as a subset of $\langle F \rangle_{\mathbb{R}}$, there exists a rational polyhedral cone $\Pi \subset F$ such that $x \in \relInt(\Pi)$ and $\dim(\Pi) = \dim(F)$. Then $\relInt(\Pi)$ is an open subset of $F$. Since the points $\{x_{n}\}$ converge to $x$ in the face $F$, there exists some number $N \in \mathbb{N}$ such that $x_{n} \in \Pi$ for all $n \geq N$. Friedman's results tells us that $\Adm$ acts on the cones $C(E_{1}, \dots, E_{k})$ with finitely many orbits. Say we choose a representative $C_{i}$ from each orbit, so we have finitely many representatives $C_{1}, \dots, C_{r}$ (note that we drop the $\{E_{1}, \dots, E_{k}\}$ part here to keep the notation simpler). By the Siegel property \ref{thm:TheSiegelProperty}, there exists finitely many elements $g \in \Adm$ such that $g(C_{i}) \; \bigcap \Pi \neq \emptyset$. Suppose these elements are $g_{i, 1}, \dots, g_{i, m_{i}}$ for $i = 1, \dots, r$. Then the following cones intersect $\Pi$:
\begin{align*}
&g_{1, 1} C_{1}, \dots, g_{1, m_{1}} C_{1} \\
&g_{2, 1} C_{2}, \dots, g_{2, m_{2}} C_{2} \\
&\vdots \\
&g_{r, 1} C_{r}, \dots, g_{r, m_{r}} C_{r}.
\end{align*}
As a result, we have a (finite) total of $m = m_{1} + \dots + m_{r}$ cones $\sigma_{l}$ of the form $C(E_{1}, \dots, E_{k})$ intersecting the cone $\Pi$. Since each cone $C(E_{1}, \dots, E_{k})$ is closed, the finite union
\begin{center}
$\displaystyle{\bigcup_{l=1}^{m} \sigma_{l}}$
\end{center}
is also closed. Recall that each $x_{n}$ is contained in some cone in the union above, so their limit point $x$ must also lie in the union, i.e.,
\begin{center}
$x \in \displaystyle{\bigcup_{l=1}^{m} \sigma_{l}} \subset \displaystyle{\bigcup \; C(E_{1}, \dots, E_{k})}$.
\end{center}
Now we have shown that $\oCurv(Y_{gen}) \subseteq \bigcup C(E_{1}, \dots, E_{k})$. Therefore,
\begin{center}
$\oCurv(Y_{gen}) = \bigcup C(E_{1}, \dots, E_{k}) = \bigcup \langle D_{1}, \dots, D_{n}, E_{1}, \dots, E_{k} \rangle_{\mathbb{R} \geq 0}$,
\end{center}
proving Proposition \ref{prop:EF-realCoefficients}.
\end{proof}

\begin{theorem} 
\label{nefImpliesSemiample}
Let $(Y_{e}, D_{e})$ be a log Calabi-Yau surface with split mixed Hodge structure. If $L$ is a nef divisor on $Y$, then $L$ is semiample.
\end{theorem}

\begin{proof}
Let $L$ be nef on $Y$. Then $L^{2} \geq 0$. If $L^{2} > 0$, then by Theorem 4.8 of \cite{F15}, the divisor $L$ is semiample. For the remainder of this proof, we suppose that $L^{2} = 0$ and $L \neq 0$. Using Riemann-Roch, we obtain
\begin{align*}
\chi(L) &= \chi(\mathcal{O}) + \displaystyle{\frac{1}{2}} L \cdot (L - K_{Y}) \\
&= 1 + \displaystyle{\frac{1}{2}} L^{2} + \displaystyle{\frac{1}{2}} (L \cdot D) \; \; \; \text{since $Y$ is rational and $K_{Y} + D = 0$} \\
&= 1 + \displaystyle{\frac{1}{2}} (L \cdot D) \; \; \; \text{using $L^{2} = 0$}.
\end{align*}
Here we note that $\chi(L) \geq 1$, since $L$ nef and $D$ effective imply that $L \cdot D \geq 0$. On the other hand, the Euler characteristic of $L$ may also be expressed as
\begin{center}
$\chi(L) = h^{0}(L) - h^{1}(L) + h^{2}(L)$.
\end{center}

By the last paragraph of the proof of Lemma \ref{YgenLnefImpliesEffective}, since $L$ is nef, we have $h^{2}(L) = 0$. Now $\chi(L) = h^{0}(L) - h^{1}(L) = 1 + \displaystyle{\frac{1}{2}} (L \cdot D)$. Next we split the inequality $L \cdot D \geq 0$ (which comes from $L$ being nef) into two subcases, and in each situation we prove that $h^{0}(L) \geq 2$. \\
\noindent {\bf Subcase (i)}. Suppose that $L \cdot D > 0$, or $L \cdot D \geq 1$. Then,
\begin{center}
$1 + \displaystyle{\frac{1}{2}}(L \cdot D) = \chi(L) = h^{0}(L) - h^{1}(L) \leq h^{0}(L)$.
\end{center}
Since $\chi(L) \in \mathbb{Z}$, we have $h^{0}(L) \geq 2$. \\
\noindent {\bf Subcase (ii)}. Suppose that $L \cdot D = 0$. We have $\chi(L) = h^{0}(L) - h^{1}(L) = 1$. We show that $h^{1}(L) \geq 1$, so $h^{0}(L) \geq 2$. Since $L \cdot D = 0$ and $L$ is nef, we have $L \cdot D_{i} = 0$ for all $i$. Then because $(Y, D)$ has split mixed Hodge structure, it follows that $\mathcal{O}_{D}(L \vert_{D}) \simeq \mathcal{O}_{D}$ (Section \ref{Background}). From the exact sequence
\begin{center}
$0 \longrightarrow \mathcal{O}_{Y}(L-D) \longrightarrow \mathcal{O}_{Y}(L) \longrightarrow \mathcal{O}_{D} \longrightarrow 0$,
\end{center}
we obtain
\begin{align*}
H^{1}(\mathcal{O}_{Y}(L)) \overset{\delta}{\longrightarrow} H&^{1}(\mathcal{O}_{D}) \longrightarrow H^{2}(\mathcal{O}_{Y}(L-D)) \\
&\vsimeq \\
&\mathbb{C}
\end{align*}
By Serre Duality, we have
\begin{align*}
h^{2}(\mathcal{O}_{Y}(L-D)) &= h^{0}(\mathcal{O}_{Y}(K_{Y} - (L-D))) \\
&= h^{0}(\mathcal{O}_{Y}(-L)) \; \; \; \text{since $K_{Y} + D = 0$} \\
&= 0.
\end{align*}
Then the map $\delta$ in the exact sequence above is surjective, so $h^{1}(L) \geq 1$.

Therefore $h^{0}(L) \geq 2$. This means that in the linear system $\vert L \vert$, there is a moving part. Writing $L = M + F$, where $M$ is the moving part and $F$ is the fixed part, we have
\begin{align*}
L^{2} &= L \cdot (M + F) \\
&= L \cdot M + L \cdot F,
\end{align*}
and $L$ nef gives $L \cdot M \geq 0$ and $L \cdot F \geq 0$. Since $L^{2} = 0$ by assumption, we obtain $L \cdot M = 0 = L \cdot F$. Now we have
\begin{align*}
L \cdot M &= (M + F) \cdot M \\
&= M^{2} + M \cdot F \\
&= 0,
\end{align*}
and $M$ is nef (since it is moving) so $M^{2} \geq 0$ and $M \cdot F \geq 0$, so $M^{2} = M \cdot F = 0$. Also,
\begin{align*}
L^{2} &= (M + F)^{2} \\
&= M^{2} + 2 M \cdot F + F^{2} \\
&= F^{2}, \; \; \; \text{since $M^{2} = 0 = M \cdot F$,}
\end{align*}
so that $F^{2} = L^{2} = 0$. We make two conclusions from the computations above. 
\begin{enumerate}[(a)]
\item The linear system $\vert M \vert$ has no fixed part, so $\vert M \vert$ is basepoint free: there exists $M^{\prime} \sim M$ such that $M$ and $M^{\prime}$ have no common components (since $M$ is moving). Then $M \cdot M^{\prime} = M^{2} = 0$, so $\Supp(M) \cap \Supp(M^{\prime}) = \emptyset$, and therefore $\vert M \vert$ is basepoint free. It follows that there exists a map $\phi_{\vert M \vert}: Y \rightarrow C$, where $C \subset \mathbb{P}^{N}$ is a curve. By Stein factorization (see Hartshorne \cite{H77}, Chapter III (11.5)), replacing $L = M + F$ by $kL = kM + kF$ for sufficiently large $k$, we may assume that $C$ is a smooth curve and $\phi$ has connected fibers.

\item Secondly we conclude that $L$ is semiample, using the results that $F^{2} = 0$ and $F \cdot M = 0$. Since $F \cdot M = 0$, the divisor $F$ is contained in a union of fibers of the map $\phi: Y \rightarrow C$. A fiber has negative semidefinite intersection matrix with kernel generated over $\mathbb{Q}$ by the class of the fiber. Therefore $kF$ is a sum of fibers for some $k > 0$. Then $k^{\prime}F$ is basepoint free for some $k^{\prime} > 0$ such that $k \vert k^{\prime}$ by Riemann-Roch on the curve $C$. Now $k^{\prime} \cdot L = k^{\prime} \cdot M + k^{\prime} \cdot F$ is basepoint free, so $L$ is semiample.
\end{enumerate}
\end{proof}

%
%

\section{PROOF OF THE CONJECTURE}
\label{ProofOfTheConjecture}

\begin{theorem} 
\label{coneConjectureYgen}
The cone conjecture for $Y_{gen}$ holds. That is, the group $\Adm$ acts on $\Nefe(Y_{gen})$ with a rational polyhedral fundamental domain.
\end{theorem}

\begin{proof}
First, assume $n \geq 3$. By Friedman's result (Corollary \ref{cor:FriedmanMinus1}), the group $\Adm$ acts on the set of finite collections of disjoint interior $(-1)$-curves with finitely many orbits. Let $C_{1}^{\prime}, \dots, C_{r}^{\prime}$ be representatives for the finitely many orbits of $\Adm$ on the set of cones $C^{\prime}$. Let $\Pi = \Conv(C_{1}^{\prime}, \dots, C_{r}^{\prime})$. Then $\Pi$ is rational polyhedral because the cones $C^{\prime}_{i}$ are, by Proposition \ref{prop:CprimeCover} (1). Moreover $\Adm \cdot \Pi = \Nefe(Y_{gen})$ by Proposition \ref{prop:CprimeCover} (2). Therefore $\Adm$ acts on $\Nefe(Y_{gen})$ with a rational polyhedral fundamental domain by Theorem \ref{prop-def:Looijenga} of Looijenga: in our setting, the lattice $L$ is $\Pic(Y_{gen})$ and $C$ is the ample cone of $Y_{gen}$ (which is the interior of $\Nef(Y_{gen})$). Its closure $\bar{C}$ is $\Nef(Y_{gen})$. The group $\Gamma$ acting on $L$ is $\Adm$. By (\ref{nefeConvexHullNefYPicY}), $C_{+} = \Nefe(Y_{gen})$. This proves the cone conjecture for $Y_{gen}$ in the case when $D_{gen}$ has at least three components.

If the number of components $n$ of $D_{gen}$ is one or two, then we show below in Section \ref{NewExamplesMDS} that the nef cone is rational polyhedral for $Y_{e}$. Moreover in these cases, the groups $\Adm$ and the Weyl group $W$ are equal (\cite{L81}, Proposition 4.7 or \cite{F15}, Theorem 9.13). Because the action of $W$ on $\Nefe(Y_{gen})$ has fundamental domain $\Nefe(Y_{e})$ (\cite{GHK15b}, Theorem 3.2), we conclude that $\Adm = W$ acts on $\Nefe(Y_{gen})$ with the {\it rational polyhedral} fundamental domain $\Nef(Y_{e})$, proving the cone conjecture.
\end{proof}

\begin{theorem}
\label{coneConjectureYprimegen}
The cone conjecture for $Y^{\prime}_{gen}$ holds. That is, the group $\Adm$ acts on $\Nefe(Y^{\prime}_{gen})$ with a rational polyhedral fundamental domain.
\end{theorem}

\begin{remark}
We use the definition of $\Nefe(Y^{\prime}_{gen})$ as given in \ref{rem:equivalentNefYprimeDef}.
\end{remark}

\begin{proof} 
By Theorem \ref{coneConjectureYgen}, we know that the cone conjecture holds for $Y_{gen}$. Since $\Nefe(Y^{\prime}_{gen})$ is a face $F$ of $\Nefe(Y_{gen})$, by Looijenga's result (see the last statement of Theorem \ref{prop-def:Looijenga}), the cone conjecture also holds for $Y^{\prime}_{gen}$. In our setting, the normalizer $N_{\Gamma}F$ is $\Adm$ and the centralizer $Z_{\Gamma}F$ is $\{e\}$.
\end{proof}

\begin{theorem}
\label{coneConjectureYe}
Let $K$ be the kernel of the action of $\Aut(Y_{e}. D_{e})$ on $\Pic(Y_{e})$. Then $\Aut(Y_{e}, D_{e}) / K$ acts on $\Nefe(Y_{e})$ with a rational polyhedral fundamental domain.
\end{theorem}

\begin{remark}
The proof of Theorem \ref{coneConjectureYe} is similar to the argument of Sterk for $K3$ surfaces (see \cite{S85}).
\end{remark}

\begin{proof}[Proof of Theorem \ref{coneConjectureYe}]
By Theorem \ref{coneConjectureYgen}, the group $\Adm$ acts on $\Nefe(Y_{gen})$ with a rational polyhedral fundamental domain. Choose a rational point $y \in \Int(\Nefe(Y_{gen}))$ such that $y$ has trivial stabilizer in $\Adm$. Then by \cite{L81} Application 4.14, we obtain a rational polyhedral fundamental domain $\sigma(y)$ defined as follows:
\begin{center}
$\sigma(y) = \sigma := \{ x \in \Nefe(Y_{gen}) \; \vert \; \gamma x \cdot y \geq x \cdot y \text{ for all }\gamma \in \Adm \}$.
\end{center}

Let $\gamma = s_{\alpha}$, the reflection associated to a simple root $\alpha = [C]$ where $C \subset Y_{e} \setminus D_{e}$ is a $(-2)$-curve. Because $ s_{\alpha}(x) = x + (x \cdot \alpha) \alpha$, the condition
\begin{equation}
\gamma x \cdot y \geq x \cdot y
\end{equation}
is equivalent to
\begin{align*}
s_{\alpha}(x) \cdot y \geq x \cdot y &\Longleftrightarrow (x + (x \cdot \alpha) \alpha) \cdot y \geq x \cdot y \\
&\Longleftrightarrow  x \cdot y + (x \cdot \alpha)(\alpha \cdot y) \geq x \cdot y \\
&\Longleftrightarrow  (x \cdot \alpha)(\alpha \cdot y) \geq 0.
\end{align*}
Because $\alpha$ is effective and $y$ is ample (since $y \in \Int(\Nef(Y_{e}))$, which is the ample cone), the intersection $(\alpha \cdot y)$ is positive. Then $(x \cdot \alpha)(\alpha \cdot y) \geq 0$ if and only if $(x \cdot \alpha) \geq 0$. In particular, this shows the following:
\begin{center}
$\sigma \subset \Nefe(Y_{gen}) \cap (\alpha \geq 0 \; \forall \; \alpha \in \Delta) = \Nefe(Y_{e})$,
\end{center}
where $\Delta$ above denotes the simple roots (see Definition \ref{def:simpleRoots}) and the equality follows from the description of the nef cone in \cite{GHK15b}, Lemma 2.15.

The following statements are true:
\begin{enumerate}
\item $\sigma$ is rational polyhedral (\cite{L14}, Application 4.14) and $\sigma \subset \Nefe(Y_{e})$, as shown above;
\item $\Adm = W \rtimes \Aut(Y_{e}, D_{e})/K$ (\cite{GHK15b}, Theorem 5.1 or \cite{F15}, Theorem 9.6);
\item $\Nefe(Y_{e})$ is a fundamental domain for the action of $W$ on $\Nefe(Y_{gen})$ (\cite{GHK15b}, Theorem 3.2), and by the Torelli theorem (\cite{GHK15b}, Theorem 1.8), $\Aut(Y_{e}, D_{e})/K \leq \Adm$ is the normalizer of $\Nefe(Y_{e})$.
\end{enumerate}

Next, we show how the three statements above imply that $\sigma$ is a rational polyhedral fundamental domain for the action of $\Aut(Y_{e}, D_{e}) / K$ on $\Nefe(Y_{e})$. Let $g$ be an element of $\Adm$. By (2) above, there exist unique $w \in W$ and $\theta \in \Aut(Y_{e}, D_{e}) / K$ such that $g = w \theta$. We claim that the following inclusion holds:
\begin{equation}
\label{eqn:inclusion}
\bigl( g \sigma \bigr) \; \bigcap \; \Nefe(Y_{e}) \subset \theta \sigma
\end{equation}
To see why, let $\mathcal{C}$ be the cone
\begin{center}
$\mathcal{C} = (\alpha \geq 0 \text{ for } \alpha \in \Delta)$,
\end{center}
which is the fundamental chamber for the action of $W$ on $\Pic(Y) \otimes_{\mathbb{Z}} \mathbb{R}$, and we have $\sigma \subset \Nefe(Y_{e}) \subset \mathcal{C}$.
From above, $g = w \theta$. Let $x \in (g \sigma) \; \bigcap \; \Nefe(Y_{e})$. The group $\Aut(Y_{e}, D_{e}) / K$ acts on $\Nefe(Y_{e})$. Then $x \in \mathcal{C}$ and $x = gu = w \theta u$, where $u \in \sigma \subset \Nefe(Y_{e})$ and so $\theta u \in \Nefe(Y_{e}) \subset \mathcal{C}$. Thus $\theta u$ and $w \theta u$ are in $\mathcal{C}$. So $\theta u \in w \mathcal{C} \; \bigcap \; \mathcal{C} \subset \Fix(w)$ by Sterk (\cite{S85}, Lemma 1.2), which means that $w \theta u = \theta u$, i.e., $x = \theta u$. Then $x = \theta u \in \theta \sigma$, proving the inclusion \ref{eqn:inclusion}.

Finally, we want to show that $\sigma$ is a rational polyhedral fundamental domain for the action of $\Aut(Y_{e}, D_{e})/K$ on $\Nefe(Y_{e})$. By (1), $\sigma$ is rational polyhedral, so it remains to show that it is a fundamental domain. Because $\sigma$ is a fundamental domain for the action of $\Adm$ on $\Nefe(Y_{gen})$,
\begin{center}
$\displaystyle{\Nefe(Y_{gen}) = \bigcup_{g \in \Adm} \; g \sigma}$.
\end{center}
We can write 
\begin{align*}
\Nefe(Y_{gen}) \; \bigcap \; \Nefe(Y_{e}) &= \bigl( \bigcup_{g \in \Adm} \; g \sigma \bigr) \; \cap \; \Nefe(Y_{e}) \\
&= \bigcup_{g \in \Adm} \; (g \sigma \; \cap \; \Nefe(Y_{e})) \\
&= \bigcup_{\theta \in \Aut(Y_{e}, D_{e})/K} \; \theta \sigma.
\end{align*}

Here is why the last equality holds: if $g = w \theta$, then we showed above that $(g \sigma) \; \bigcap \; \Nefe(Y_{e}) \subset \theta \sigma$. If $w = 1$, then $g = \theta \in \Aut(Y_{e}, D_{e}) / K$ and
\begin{center}
$(g \sigma) \; \bigcap \; \Nefe(Y_{e}) = (\theta \sigma) \; \bigcap \; \Nefe(Y_{e}) = \theta \sigma$,
\end{center}
because $\theta$ preserves $\Nefe(Y_{e})$ and $\sigma \subset \Nefe(Y_{e})$. Moreover, because
\begin{center}
$\Int \bigl( g_{1} \sigma \; \bigcap \; g_{2} \sigma \bigr) = \emptyset \; \forall \; g_{1}, g_{2} \in \Adm$,
\end{center}
it follows that the same statement holds for $g_{1}, g_{2}$ in the smaller group $\Aut(Y_{e}, D_{e})/K$. That is,
\begin{center}
$\Int \bigl( g_{1} \sigma \; \bigcap \; g_{2} \sigma \bigr) = \emptyset \; \forall \; g_{1}, g_{2} \in \Aut(Y_{e}, D_{e})/K$,
\end{center}
which is the second property in the definition of a fundamental domain. Therefore we have shown, using conditions (1), (2), and (3), that $\sigma$ is a rational polyhedral fundamental domain for the action of $\Aut(Y_{e}, D_{e})/K$ on $\Nefe(Y_{e})$.
\end{proof}

%
%

\section{NEW EXAMPLES OF MORI DREAM SPACES}
\label{NewExamplesMDS}

\begin{definition} 
\label{def:MoriDreamSpace}
(cf. \cite{HK00}, Definition 1.10) Let $Y$ be a smooth projective surface. Then $Y$ is a Mori Dream Space if the following are true:
\begin{enumerate}
\item $\Pic(Y) \otimes \mathbb{R} = N^{1}(Y)$;
\item $\Nef(Y)$ is rational polyhedral; and 
\item If $L$ is a nef divisor on $Y$, then $L$ is semiample.
\end{enumerate}
\end{definition}

\begin{theorem}
\label{thm:proofn1-6}
A log Calabi-Yau surface $(Y_{e}, D_{e})$ with split mixed Hodge structure in which the boundary $D_{e}$ consists of no more than six components has a rational polyhedral cone of curves.
\end{theorem}

In addition, for each such surface, we give an explicit description of the cone of curves. We note that some similar calculations for $n \leq 5$ appear in \cite{L81}.

\begin{remark}
When the cone of curves of $Y$ has finitely many generators, it is automatically closed. Because the cones we describe below are all rational polyhedral, we have $\Curv(Y) = \oCurv(Y)$.
\end{remark}

In this next part, we only consider log Calabi-Yau surfaces with the split mixed Hodge structure. We will show that each surface $Y$ described for each $n \leq 6$ is the surface $Y_{e}$ in the given deformation type with the split mixed Hodge structure (that is, the period point $\phi$ given by $\phi(x) = 1$ for all $x \in \langle D_{1}, \dots, D_{n} \rangle_{\mathbb{Z}}^{\bot})$.

\begin{lemma}
Let $(Y, D)$ be a log Calabi-Yau surface and suppose that $\langle D_{1}, \dots, D_{n} \rangle_{\mathbb{Z}}^{\bot}$ is generated by classes of curves $C \subset Y \setminus D$. Then $\phi(x) = 1$ for all $x \in \langle D_{1}, \dots, D_{n} \rangle_{\mathbb{Z}}^{\bot}$.
\end{lemma}

\begin{proof}
Using the notation $\theta: \Pic^{0}(D) \xrightarrow{\sim} \mathbb{C}^{\ast}$ from Definition \ref{def:periodPoint}, we recall that $\phi([C]) = \theta(\mathcal{O}_{Y}(C) \vert_{D})$. Because $C \bigcap D = \emptyset$, the restriction $\mathcal{O}_{Y}(C) \vert_{D} = \mathcal{O}_{D}$ is the trivial line bundle on $D$. Then $\phi([C]) = 1$. From our assumption it follows that $\phi(x) = 1$ for all $x \in \langle D_{1}, \dots, D_{n} \rangle_{\mathbb{Z}}^{\bot}$.
\end{proof}

The lemma applies in our situation, because for the surfaces we describe in cases $n \leq 6$, the lattice $\langle D_{1}, \dots, D_{n} \rangle_{\mathbb{Z}}^{\bot}$ is generated by the classes of $(-2)$-curves $C$ in $Y \setminus D$.

\begin{remark}
We cover every deformation type for each $n \leq 6$ of log Calabi-Yau surfaces such that the intersection matrix $(D_{i} \cdot D_{j})$ is negative definite or negative semidefinite. This follows from two theorems below: Looijenga for $n \leq 5$ (Theorem \ref{thm:Looijenga}) and Simonetti for $n = 6$ (Theorem \ref{thm:Simonetti}).
\end{remark}

\begin{remark}
If the intersection matrix $(D_{i} \cdot D_{j})$ is not negative definite nor negative semidefinite, then $\Nef(Y)$ is rational polyhedral by the Cone Theorem (\cite{GHK15a}, Lemma 6.9).
\end{remark}

To keep the notation simple, we will use $(Y, D)$ to mean $(Y_{e}, D_{e})$ from here on in this section, unless otherwise specified. The proof of Theorem \ref{thm:proofn1-6} is split into the six cases $n=1, \dots, 6$. We use the following theorem and lemma. We note that in each of the cases considered, the number of boundary components remains the same.

\begin{theorem}
\label{thm:Looijenga} 
(\;\cite{L81} Theorem 1.1)
Let $Y$ be a smooth rational surface endowed with an anti-canonical cycle $D = D_{1} + \dots + D_{n}$ of length $n \leq 5$ such that $(D_{i} \cdot D_{j})$ is negative definite or negative semidefinite. Then there exists a sequence of blowdowns of interior $(-1)$-curves which gives a smooth rational surface $\bar{Y}$ endowed with an anticanonical cycle $\bar{D} = \bar{D}_{1} + \dots, + \bar{D}_{n}$, where $\bar{D}_{i}$ is the image of $D_{i}$ for each $i$, such that:
\begin{enumerate}
\item If $n = 1$, then $\bar{Y} \cong \mathbb{P}^{2}$ and $\bar{D}$ is a cubic curve with a node;
\item If $n = 2$, then $\bar{Y} \cong \mathbb{P}^{1} \times \mathbb{P}^{1}$ and $\bar{D}_{1}$ and $\bar{D}_{2}$ are $(1, 1)$-divisors.
\item If $n = 3$, then $\bar{Y} \cong \mathbb{P}^{2}$ and $\bar{D}$ is a triangle of lines (i.e., $\bar{D}$ is a choice of toric boundary); 
\item If $n = 4$, then $\bar{Y} \cong  \mathbb{P}^{1} \times \mathbb{P}^{1}$ and $\bar{D}$ is a square consisting of two fibers of each of the two projections $\bar{Y} \rightarrow \mathbb{P}^{1}$ (i.e., $\bar{D}$ is a choice of toric boundary); and
\item If $n = 5$, then $\bar{Y}$ is a del Pezzo surface of degree five (i.e., the blowup of four points on $\mathbb{P}^{2}$) and each component $\bar{D}_{i}$ of $\bar{D}$ is a $(-1)$-curve.
\end{enumerate}
\end{theorem}

\noindent The general blowup picture is shown in Figure \ref{fig:general-bus}: each boundary component is linked to a ``chain" of a single $(-1)$-curve, followed by arbitrarily many $(-2)$-curves. A general chain is shown in Figure \ref{fig:general-chain}.

\begin{figure}
\centering
\includegraphics[scale=0.4]{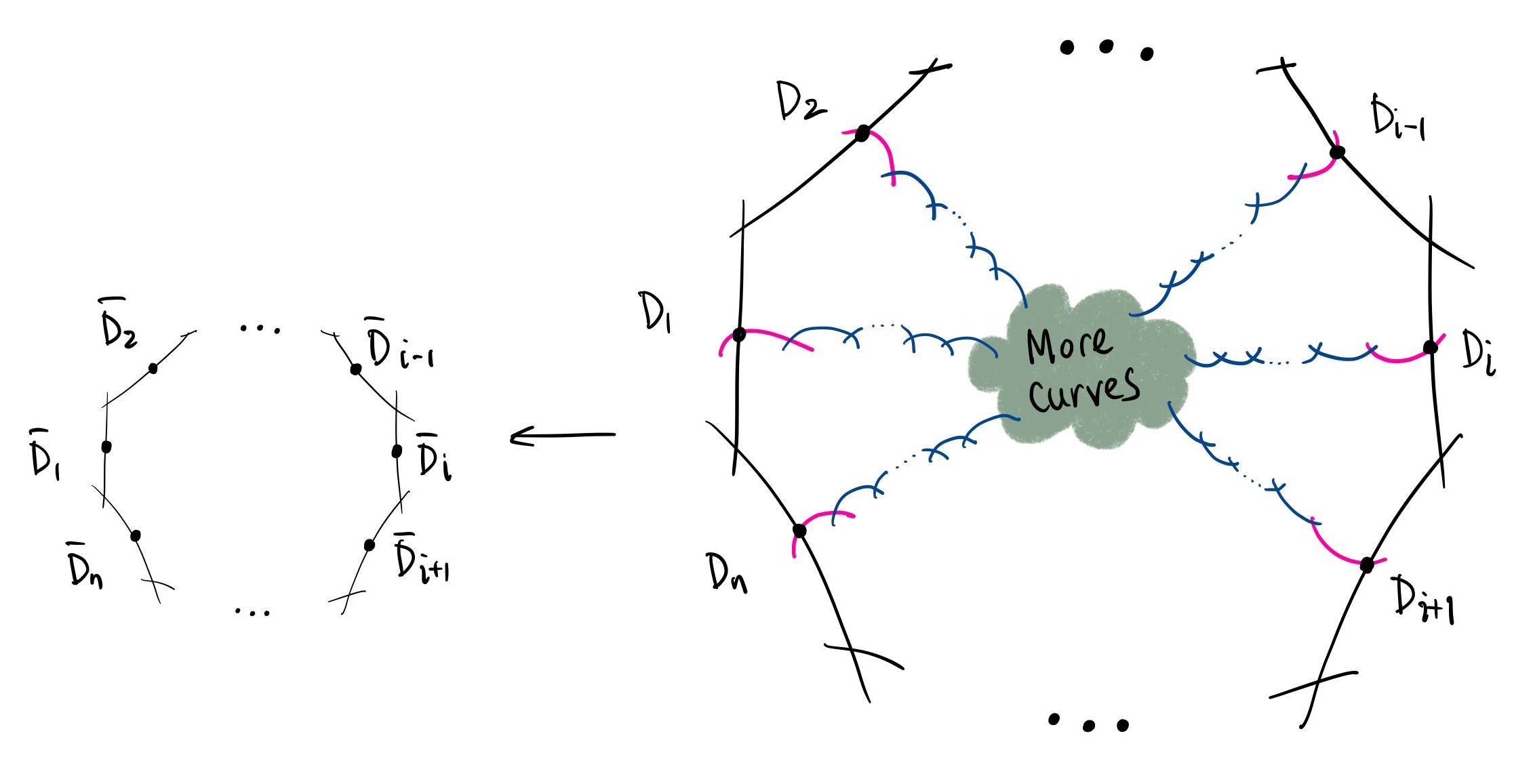}
\caption[Some number of blowups of $(\bar{Y}, \bar{D})$]{This drawing shows arbitrarily many blowups of the surface $(\bar{Y}, \bar{D})$. The blowup at each point creates a chain of one $(-1)$-curve, intersecting the boundary component at one point, followed by some number $(-2)$-curves that lead to a common ``central region". The center consists of additional curves.}
\label{fig:general-bus}
\end{figure}

\begin{figure}
\centering
\includegraphics[scale=0.4]{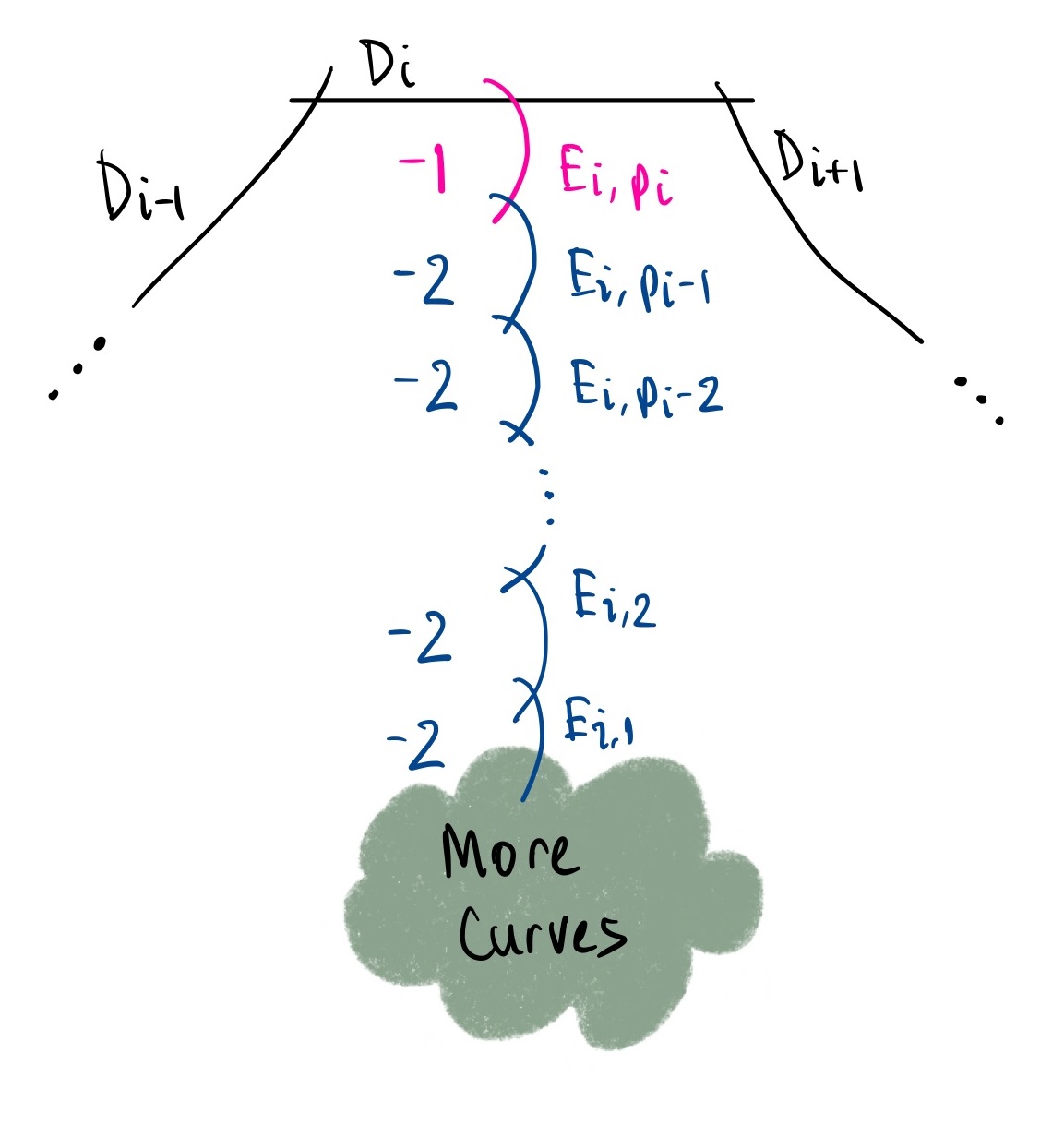}
\caption[A general ``chain" after some number of blowups]{This drawing depicts a general ``chain" that forms after arbitrarily many blowups. For a boundary of length $n$, the value of $i$ ranges from 1 to $n$.}
\label{fig:general-chain}
\end{figure}

\begin{lemma}
\label{lem:curvYgenerators}
Let $Y$ be a smooth projective complex surface. Let $\mathcal{B}$ be a basis for $N_{1}(Y)$ consisting of irreducible curves. Suppose the dual basis may be expressed as effective combinations of a set $\mathcal{C}$ of curves. Then $\oCurv(Y) = \langle \mathcal{B} \; \cup \; \mathcal{C} \rangle_{\mathbb{R} \geq 0}$.
\end{lemma}

\begin{proof}
Let $C \subset Y$ be a curve and suppose that $C \notin \mathcal{B}$. Then $C \cdot B_{i} \geq 0$ for all $B_{i} \in \mathcal{B}$. Since $B_{i}^{\ast}$ is an effective linear combination of elements in $\mathcal{C}$ and $C \cdot B_{i} \geq 0$, it follows that $C = \sum (C \cdot B_{i}) B_{i}^{\ast}$ belongs to $\langle \; \mathcal{C} \; \rangle_{\mathbb{R}_{\geq 0}}$.
\end{proof}

\subsection*{Number of boundary components $n=1$.}
\label{subsec:n1}
Let $\bar{Y} = \mathbb{P}^{2}$ with a rational nodal curve $\bar{D}_{1}$ and a flex point $q$. In coordinates, we may take
\begin{center}
$\bar{D}_{1}: (X_{0}X_{2}^{2} = X_{1}^{2}(X_{1} + X_{0})) \subseteq \mathbb{P}^{2}_{(X_{0}: X_{1}: X_{2})}$ and $q = (0:0:1)$.
\end{center}
We denote the tangent line at point $q$ by $\bar{F}$, and we blow up the point $q$ some number $p_{1}$ of times.

\begin{figure}
\centering
\includegraphics[scale=0.4]{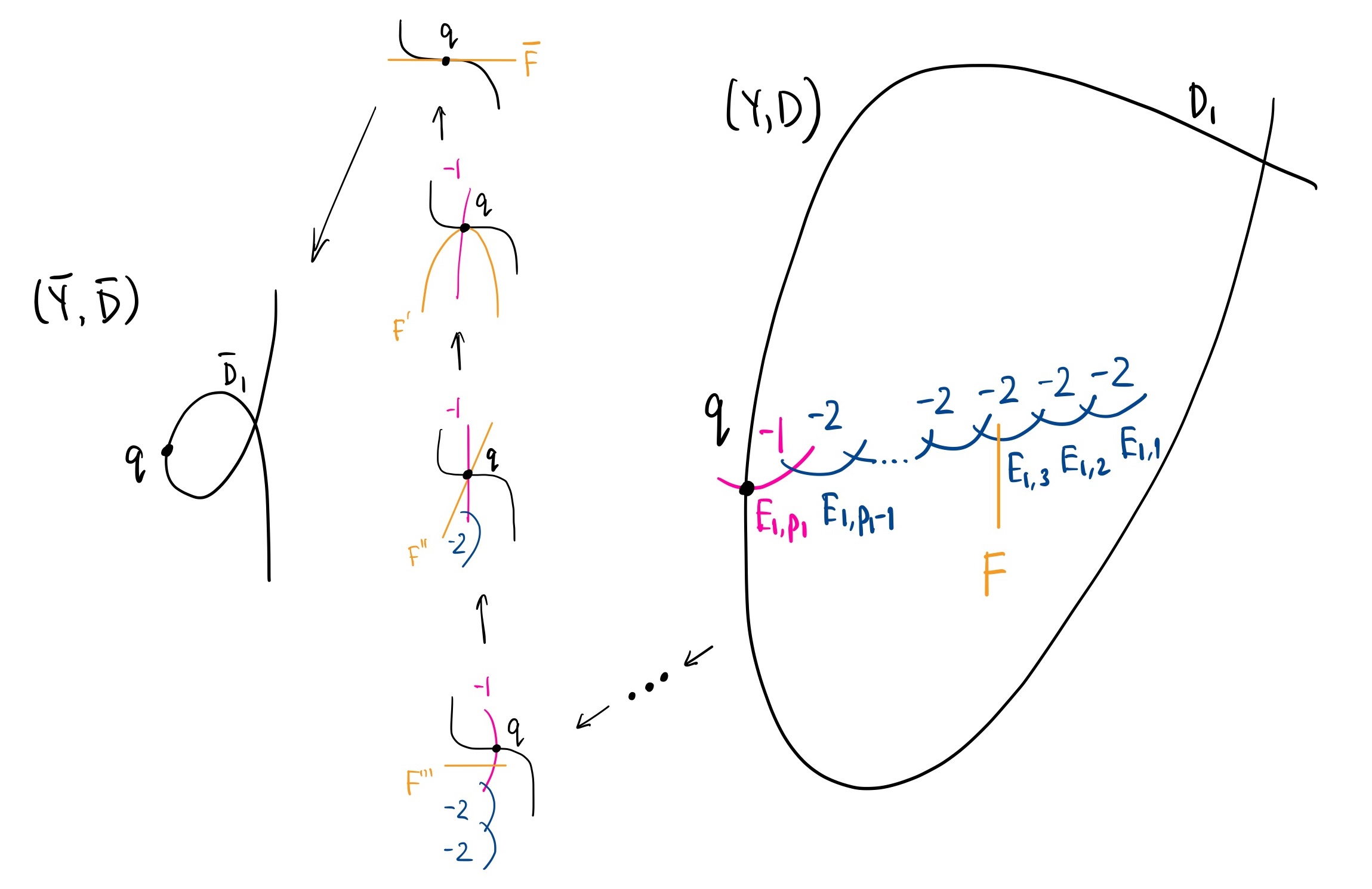}
\caption[The general picture for $n=1$]{The drawing on the far left shows $(\bar{Y}, \bar{D})$ before any blowups. The middle drawing shows the first three blowups at the point $q$, and the figure on the right depicts $(Y, D)$ after arbitrarily many, say $p_{1}$, blowups at $q$.}
\label{fig:n1-bus}
\end{figure}

A basis for $\Pic(Y)$ is
\begin{center}
$\mathcal{B}_{1} = \{E_{1, j}, F \; \vert \; 1 \leq j \leq p_{1}\}$.
\end{center}
and its dual basis $\mathcal{B}^{\ast}_{1}$ consists of the following elements:
\begin{align*}
E_{1, p_{1}}^{\ast} &= D_{1} \\
E_{1, p_{1}-1}^{\ast} &=  D_{1} + E_{1, p_{1}} \\
E_{1, p_{1}-2}^{\ast} &= D_{1} + 2E_{1, p_{1}} + E_{1, p_{1}-1} \\
E_{1, p_{1}-3}^{\ast} &= D_{1} + 3E_{1, p_{1}} + 2E_{1, p_{1}-1} + E_{1, p_{1}-2} \\
\vdots \\
E_{1, j}^{\ast} &= D_{1} + (p_{1} - j) E_{1, p_{1}} + (p_{1} - j - 1)E_{1, p_{1}-1} + \cdots + 2E_{1, j+2} + E_{1, {j+1}} \; \text{ for } 3 \leq j \leq p_{1};
\end{align*}
and
\begin{align*}
E_{1, 2}^{\ast} &= 4E_{1, p_{1}} + 4E_{1, p_{1}-1} + \cdots + 4E_{1, 4} + 4E_{1, 3} + 2E_{1, 2} + E_{1,1} + 2F \\
E_{1, 1}^{\ast} &= 2E_{1, p_{1}} + 2E_{1, p_{1}-1} + \cdots + 2E_{1, 4} + 2E_{1, 3} +  E_{1, 2} + F \\
F^{\ast} &= 3E_{1, p_{1}} + 3E_{1, p_{1}-1} + \cdots + 3E_{1, 4} + 3E_{1, 3} + 2E_{1, 2} + E_{1, 1} + F
\end{align*}
By Lemma \ref{lem:curvYgenerators}, we can describe the cone of curves as follows:
\begin{center}
$\oCurv(Y) = \langle D_{1}, E_{1, j}, F \; \vert \; 1 \leq j \leq p_{1} \rangle_{\mathbb{R}_{\geq 0}}$
\end{center}

\subsection*{Number of boundary components $n=2$.}
\label{subsec:n2}
Let $\bar{Y}$ be the Hirzebruch surface $\mathbb{F}_{2}$ with two smooth curves $\bar{D}_{1}$ and $\bar{D}_{2}$ in the linear system $\vert B + 2A \vert$. Here $B$ denotes the negative section of the $\mathbb{P}^{1}$ fibration $\mathbb{F}_{2} \rightarrow \mathbb{P}^{1}$ and $A$ denotes the fiber. We may assume that the curves $\bar{D}_{1}$ and $\bar{D}_{2}$ intersect transversely. We fix two points $q_{i} \in \bar{D}_{i}$ where $i = 1, 2$, such that the points lie on a common fiber $\bar{F}_{1}$, and let $\bar{F}_{2}$ be the $(-2)$-curve (see Figure \ref{fig:n2-Fcurves}). Then blow up at the points $q_{i}$ some number of times, which is shown in Figure \ref{fig:n2-bus} (our notation is that we blow up a total of $p_{i}$ times at points $q_{i}$ for $i = 1, 2$). The curves $F_{i}$ are the strict transforms of $\bar{F}_{i}$ for $i = 1, 2$.

\begin{figure}
\centering
\includegraphics[scale=0.3]{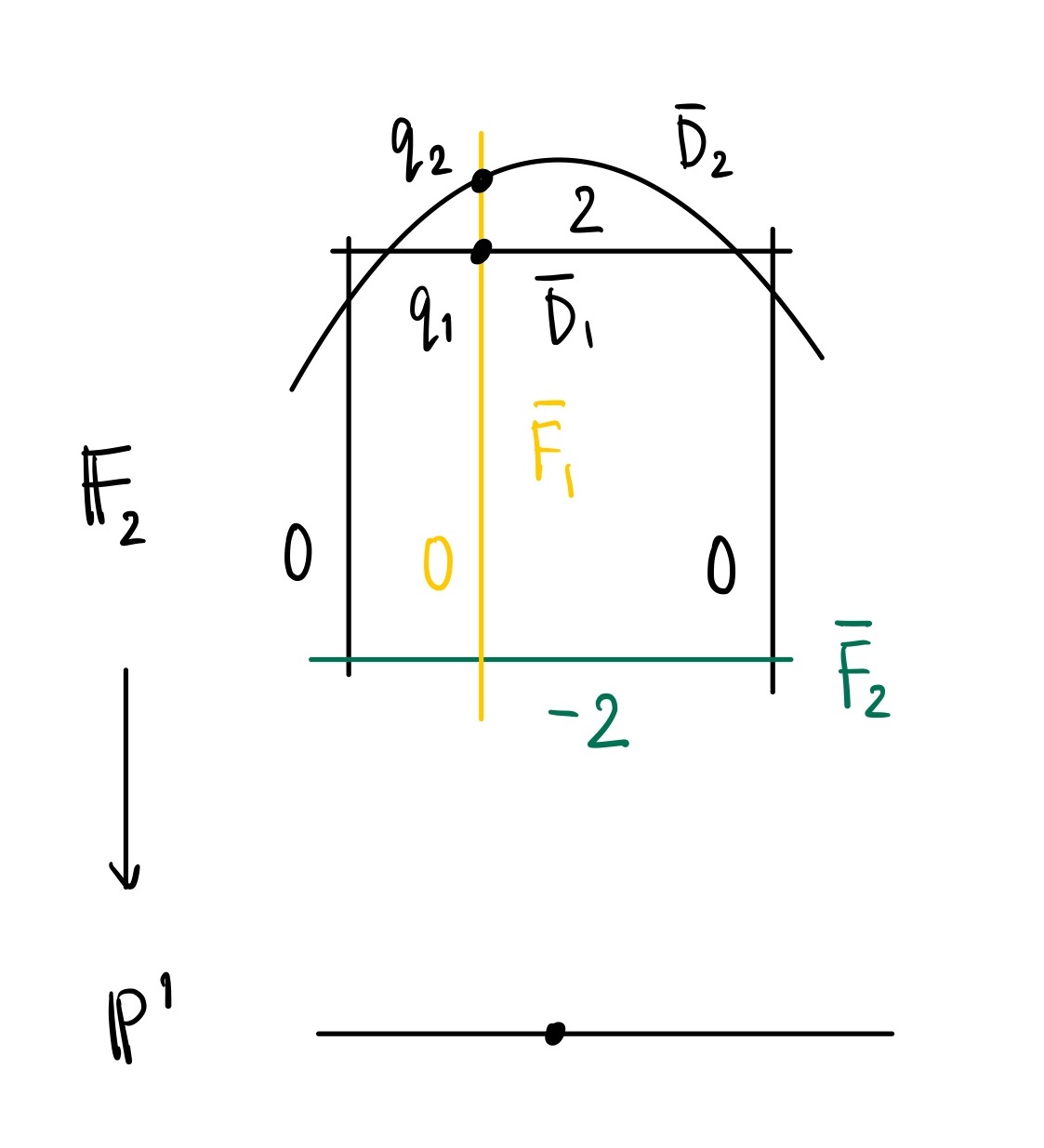}
\caption[The curves $\bar{F}_{1}$ and $\bar{F}_{2}$ in case $n=2$]{This drawing shows the curves $\bar{F}_{1}$ and $\bar{F}_{2}$ in the case $n=2$.}
\label{fig:n2-Fcurves}
\end{figure}

\begin{figure}
\centering
\includegraphics[scale=0.5]{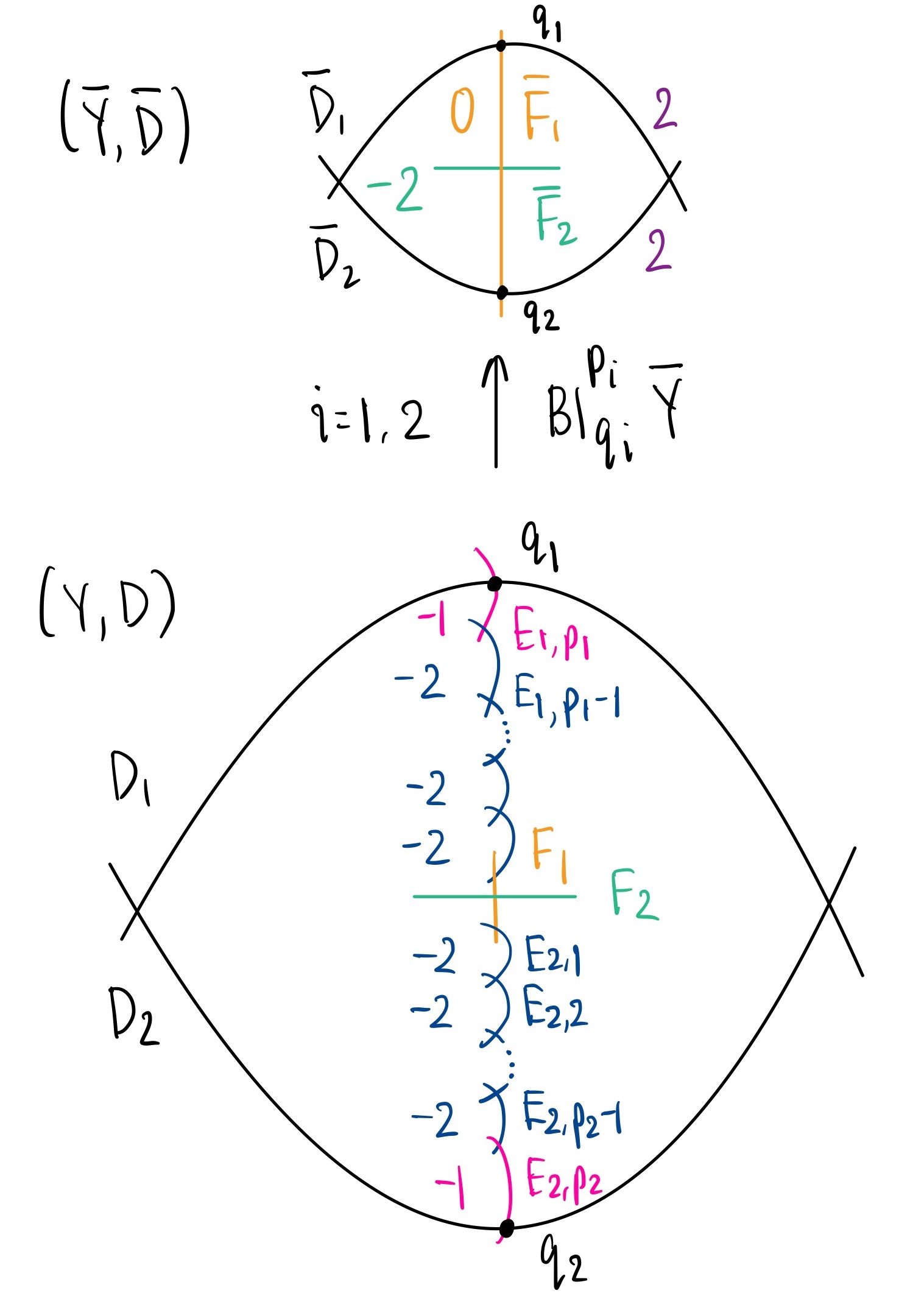}
\caption[The general picture for $n=2$]{The blowup of $(\bar{Y}, \bar{D})$ at the two points $q_{1}$ and $q_{2}$, a total of $p_{1}$ and $p_{2}$ times respectively, results in $(Y, D)$.}
\label{fig:n2-bus}
\end{figure}

A basis $\mathcal{B}_{2}$ for $\Pic(Y)$ is given by
\begin{center}
$\mathcal{B}_{2} = \{ E_{i, j}, F_{i} \; \vert \; i = 1, 2 \text{ and } 1 \leq j \leq p_{i}\}$.
\end{center}

The dual basis $\mathcal{B}^{\ast}_{2}$ consists of the following elements:
\begin{center}
$\mathcal{B}^{\ast}_{2} = \{ E_{i, j}^{\ast}, F_{1}^{\ast}, F_{2}^{\ast} \; \vert \;  i = 1, 2 \text{ and } 1 \leq j \leq p_{i} \}$,
\end{center}
where for $i = 1, 2$,
\begin{align*}
E_{i, p_{i}}^{\ast} &= D_{i} \\
E_{i, p_{i} - 1}^{\ast} &= D_{i} + E_{i, p_{i}} \\
E_{i, p_{i} - 2}^{\ast} &= D_{i} + 2E_{i, p_{i}} + E_{i, p_{i} - 1} \\
&\vdots \\
E_{i, 1}^{\ast} &= D_{i} + (p_{i}-1) E_{i, p_{i}} + (p_{i} - 2) E_{i, p_{i}-1} + \cdots + 2E_{i, 3} + E_{i, 2}
\end{align*}
and
\begin{align*}
F_{1}^{\ast} &= D_{i} + p_{i}E_{i, p_{i}} + (p_{i} - 1)E_{i, p_{i} - 1} + \cdots + 2E_{i, 2} + E_{i, 1} \; \text{ for } i = 1 \text{ or } 2,
\end{align*}
and 
\begin{align*}
F_{2}^{\ast} &= D_{i} + (p_{i} + 1)E_{i, p_{i}} + p_{i} E_{i, p_{i} - 1} + \cdots + 2E_{i, 1} + F_{1} \; \text{ for } i = 1 \text{ or } 2.
\end{align*}
By Lemma \ref{lem:curvYgenerators}, we can describe the cone of curves as follows:
\begin{center}
$\oCurv(Y) = \langle D_{i}, E_{i, j}, F \; \vert \; i = 1, 2 \text{ and } 1 \leq j \leq p_{i} \rangle_{\mathbb{R}_{\geq 0}}$.
\end{center}

\subsection*{Number of boundary components $n=3$.}
\label{subsec:n3}

Let $\bar{Y} = \mathbb{P}^{2}$ with $\bar{D} = \bar{D}_{1} + \bar{D}_{2} + \bar{D}_{3}$ its toric boundary, which is the union of three lines. Fix three collinear points $q_{i} \in \bar{D}_{i}$ where $i = 1, 2, 3$ and blow them up some number of times. Let $F$ be the strict transform of the line $\bar{F}$ passing through the three points (see Figure \ref{fig:n3-pBUs}).

\begin{figure}
\centering
\includegraphics[scale=0.25]{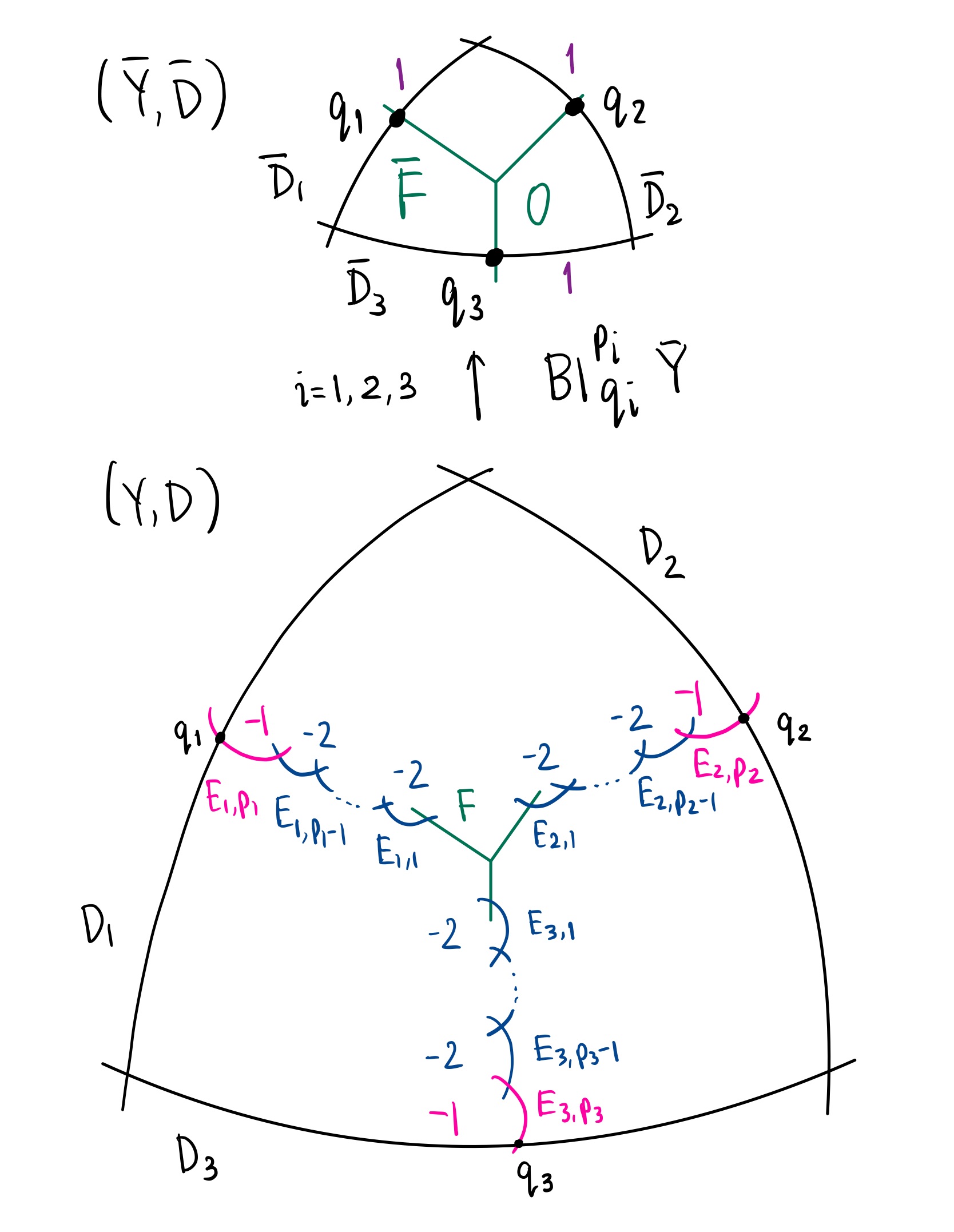}
\caption[The general picture for $n=3$]{After $(\bar{Y}, \bar{D})$ is blown up arbitrarily many $(p_{i})$ times at each point $q_{i}$, the resulting pair is $(Y, D)$.}
\label{fig:n3-pBUs}
\end{figure}

A basis $\mathcal{B}_{3}$ for $\Pic(Y)$ is given by
\begin{center}
$\mathcal{B}_{3} = \{ E_{i, j}, F \; \vert \; 1 \leq i \leq 3 \text{ and } 1 \leq j \leq p_{i} \}$.
\end{center}
The dual basis $\mathcal{B}^{\ast}_{3}$ consists of the elements, for $i = 1, 2, 3$:
\begin{align*}
E_{i, p_{i}}^{\ast} &= D_{i} \\
E_{i, p_{i} - 1}^{\ast} &= D_{i} + E_{i, p_{i}} \\
E_{i, p_{i} - 2}^{\ast} &= D_{i} + 2E_{i, p_{i}} + E_{i, p_{i} - 1} \\
&\vdots \\
E_{i, 1}^{\ast} &= D_{i} + (p_{i}-1)E_{i, p_{i}} + (p_{i} - 2)E_{i, p_{i}-1} + \cdots + 2E_{i, 3} + E_{i, 2}
\end{align*}
and
\begin{center}
$F^{\ast} = D_{i} + p_{i}E_{i, p_{i}} + (p_{i} - 1)E_{i, p_{i}-1} + \cdots + 2E_{i, 2} + E_{i, 1}$.
\end{center}
By Lemma \ref{lem:curvYgenerators}, we can describe the cone of curves as follows:
\begin{center}
$\oCurv(Y) = \langle D_{i}, E_{i, j}, F \; \vert \; 1 \leq i \leq 3$ and $1 \leq j \leq p_{i} \rangle_{\mathbb{R}_{\geq 0}}$.
\end{center}

\subsection*{Number of boundary components $n=4$.}
\label{subsec:n4}

Let $\bar{Y} = \mathbb{P}^{1} \times \mathbb{P}^{1}$ with its toric boundary, which is the union of two fibers of each of the two projections $\mathbb{P}^{1} \times \mathbb{P}^{1} \rightarrow \mathbb{P}^{1}$. Fix four points $q_{i} \in \bar{D}_{i}$ where $i = 1, \dots, 4$ such that $q_{1}$ and $q_{3}$ lie on a fiber $\bar{F}_{1}$ of the first projection and $q_{2}$ and $q_{4}$ lie on a fiber $\bar{F}_{2}$ of the second projection. Then blow them up some number of times (see Figure \ref{fig:n4-pBUs}).

\begin{figure}
\centering
\includegraphics[scale=.25]{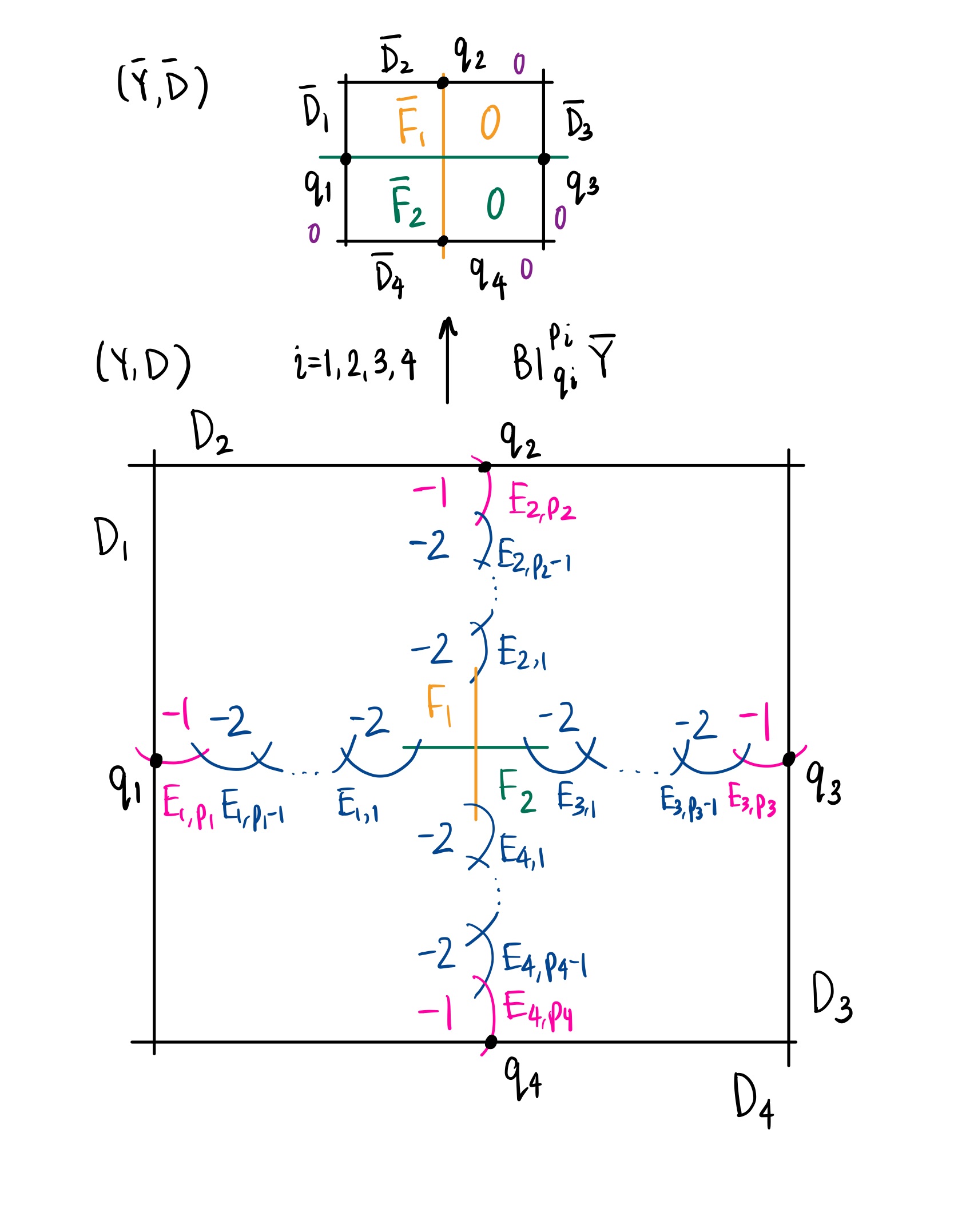}
\caption[The general picture for $n=4$]{This shows the blowup of the toric pair $(\bar{Y}, \bar{D})$ at the points $q_{i}$. At each point there are arbitrarily many ($p_{i}$) blowups, resulting in $(Y, D)$.}
\label{fig:n4-pBUs}
\end{figure}

A basis for $\Pic(Y)$ is
\begin{center}
$\mathcal{B}_{4} = \{E_{i,j} F_{1}, F_{2} \; \vert \; 1 \leq i \leq 4 \text{ and } 1 \leq j \leq p_{i}\}$.
\end{center}
The dual basis $\mathcal{B}^{\ast}_{4}$ consists of the following elements:
\begin{align*}
E_{i, p_{i}}^{\ast} &= D_{i} \\
E_{i, p_{i}-1}^{\ast} &= D_{i} + E_{i, p_{i}} \\
E_{i, p_{i}-2}^{\ast} &= D_{i} + 2E_{i, p_{i}} + E_{i, p_{i} - 1} \\
&\vdots \\
E_{i, 1}^{\ast} &= D_{i} + (p_{i}-1) E_{i, p_{i}} + (p_{i} - 2) E_{i, p_{i}-1} + \cdots + 2E_{i, 3} + E_{i, 2}
\end{align*}
and for each of $F_{j}$ where $j = 1, 2$, there are two (linearly equivalent) possibilities:
\begin{center}
$F_{j}^{\ast} = D_{i} + p_{i}E_{i, p_{i}} + (p_{i}-1)E_{i, p_{i}-1} + \cdots + 2E_{i, 2} + E_{i, 1}$
\end{center}
with $i= 2$ or $4$ for $j=1$ and $i=1$ or $3$ for $j=2$ . By Lemma \ref{lem:curvYgenerators}, we can describe the cone of curves as follows:
\begin{center}
$\oCurv(Y) = \langle D_{i}, E_{i, j}, F_{1}, F_{2} \; \vert \; 1 \leq i \leq 4 \text{ and } 1 \leq j \leq p_{i} \rangle_{\mathbb{R}_{\geq 0}}$.
\end{center}

\subsection*{Number of boundary components $n=5$.}
\label{subsec:n5}

Let $\bar{Y}$ be the blowup of four points in general position in $\mathbb{P}^{2}$ and let $\bar{D}$ be a cycle of five $(-1)$-curves. The surface $\bar{Y}$ contains ten $(-1)$-curves:
\begin{enumerate}
\item Four are exceptional curves $E_{i}$ from blowing up the points $p_{i}$, for $i = 1, 2, 3, 4$.
\item Six (obtained by 6 = $4 \choose 2$) are strict transforms $l^{\prime}_{ij}$ of lines $l_{ij}$ defined by points $p_{i}$ and $p_{j}$.
\end{enumerate}
The process of blowing up points $p_{i}$ for $i=1, \dots, 4$ on $\mathbb{P}^{2}$ to obtain a surface with ten curves is shown in Figure \ref{fig:n5-triangleBU}. Taking the dual of this figure (see Figure \ref{fig:n5-dualToPetersen}), we choose a pentagon inside and rearrange vertices so that this pentagon encloses all other vertices. Then the interior vertices can be rearranged to form a star, resulting in the Petersen graph. The dual of the interior five-pointed star is a pentagon, and the dual of the outside pentagonal boundary is again a pentagon. Together, these two parts form the configuration of $(-1)$-curves on the surface $\bar{Y}$. The $\bar{F}_{i}$'s are $(-1)$-curves not contained in the boundary in $\bar{Y}$; the $\bar{F}_{i}$'s correspond to a pentagonal star. Each $\bar{F}_{i}$ intersects the boundary component $\bar{D}_{i}$, and we denote their strict transforms by $F_{i}$.

\begin{figure}
\centering
\includegraphics[scale=0.4]{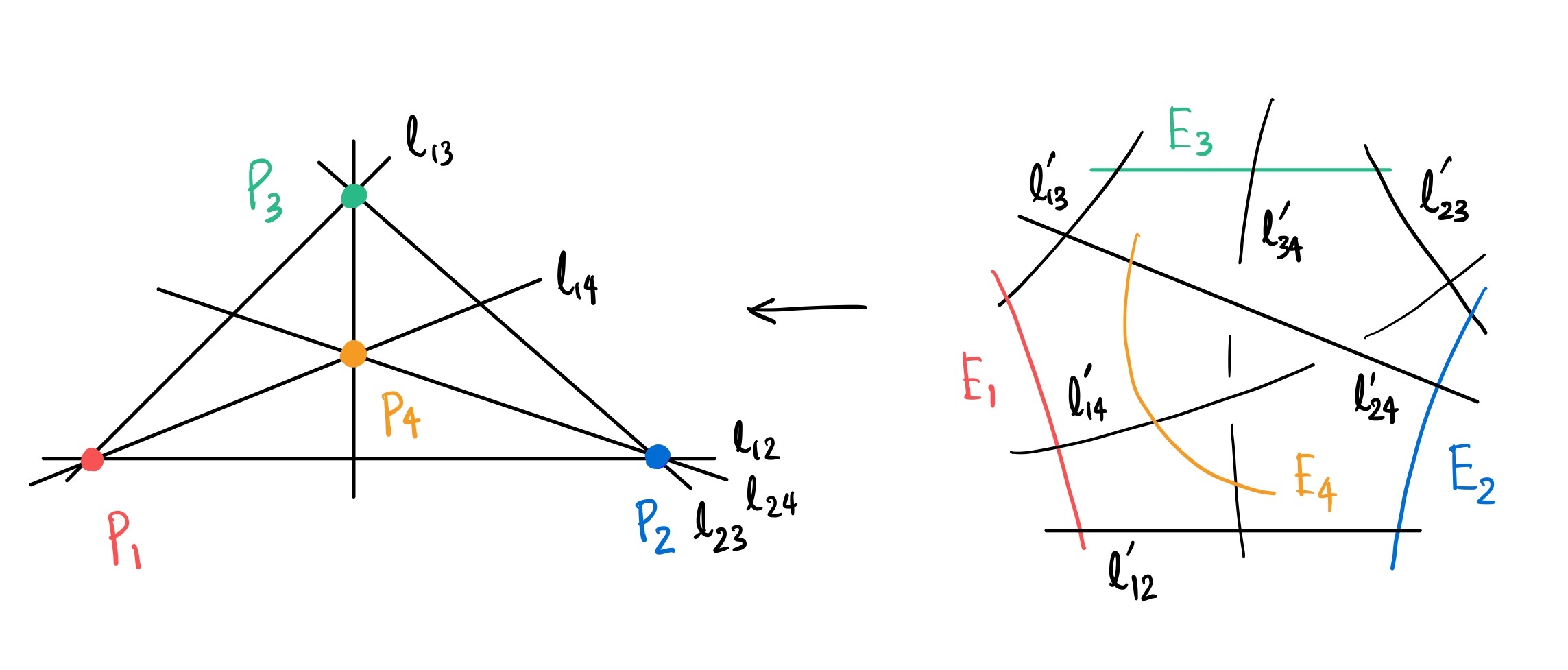}
\caption[The ten exceptional curves on $\bar{Y}$ for case $n=5$]{Blowing up once at each point $p_{1}, \dots, p_{4}$ results in the diagram above.}
\label{fig:n5-triangleBU}
\end{figure}

\begin{figure}
\centering
\includegraphics[scale=0.4]{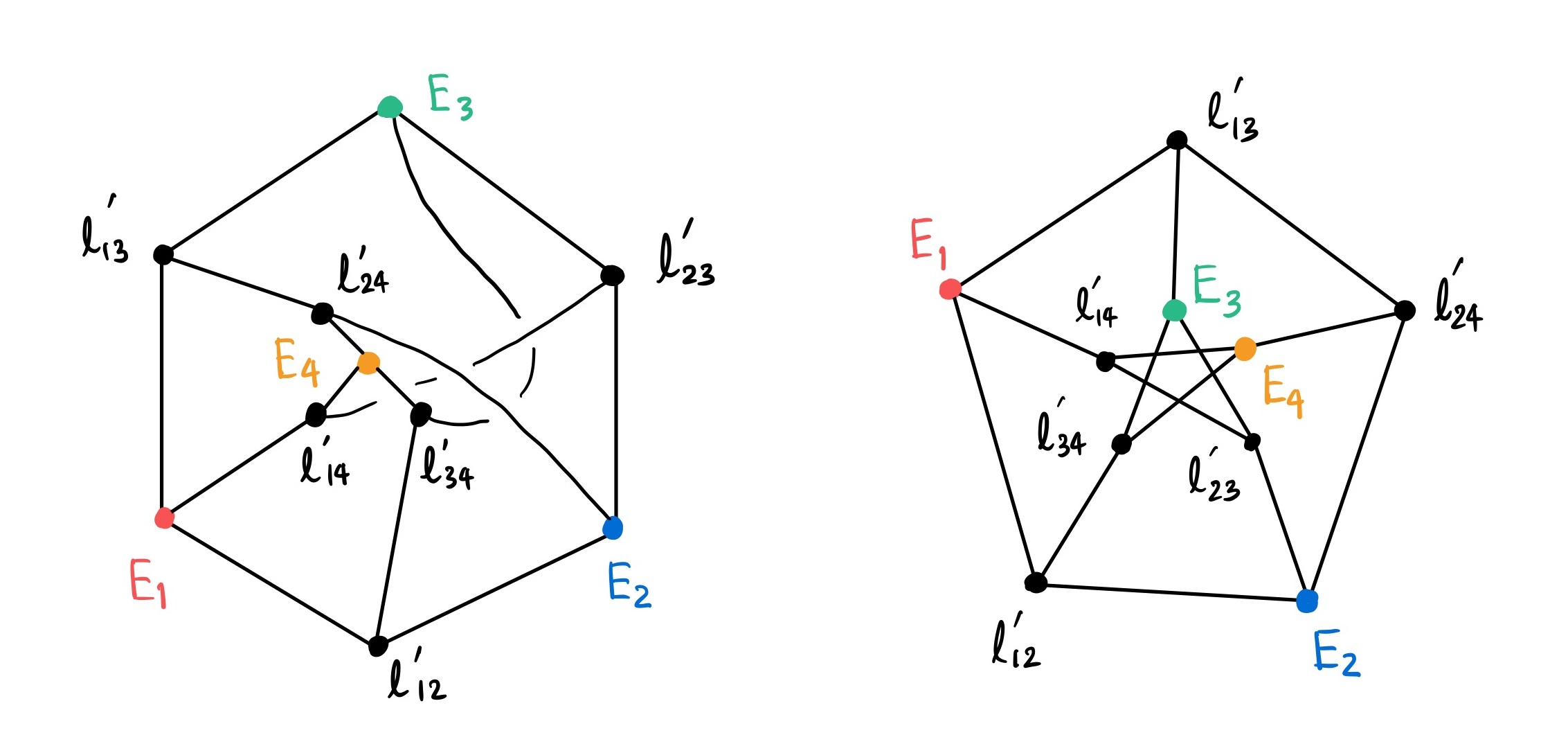}
\caption[The Petersen graph]{The dual graph of the blowup shown in Figure \ref{fig:n5-triangleBU} is drawn on the left, and it is equivalent to the Petersen graph shown on the right.}
\label{fig:n5-dualToPetersen}
\end{figure}

The remaining $(-1)$-curves on $\bar{Y}$ intersect $\bar{D}$ transversely in five points $q_{i}$ where $i = 1, \dots, 5$. Blow up these points some number of times to obtain Figure \ref{fig:n5-pBUs}. A basis $\mathcal{B}_{5}$ for $\Pic(Y)$ is the collection:
\begin{center}
$\mathcal{B}_{5} = \{ E_{i, j}, F_{i} \; \vert \; 1 \leq i \leq 5 \text{ and } 1 \leq j \leq p_{i}\}$.
\end{center}

\begin{figure}
\centering
\includegraphics[scale=0.5]{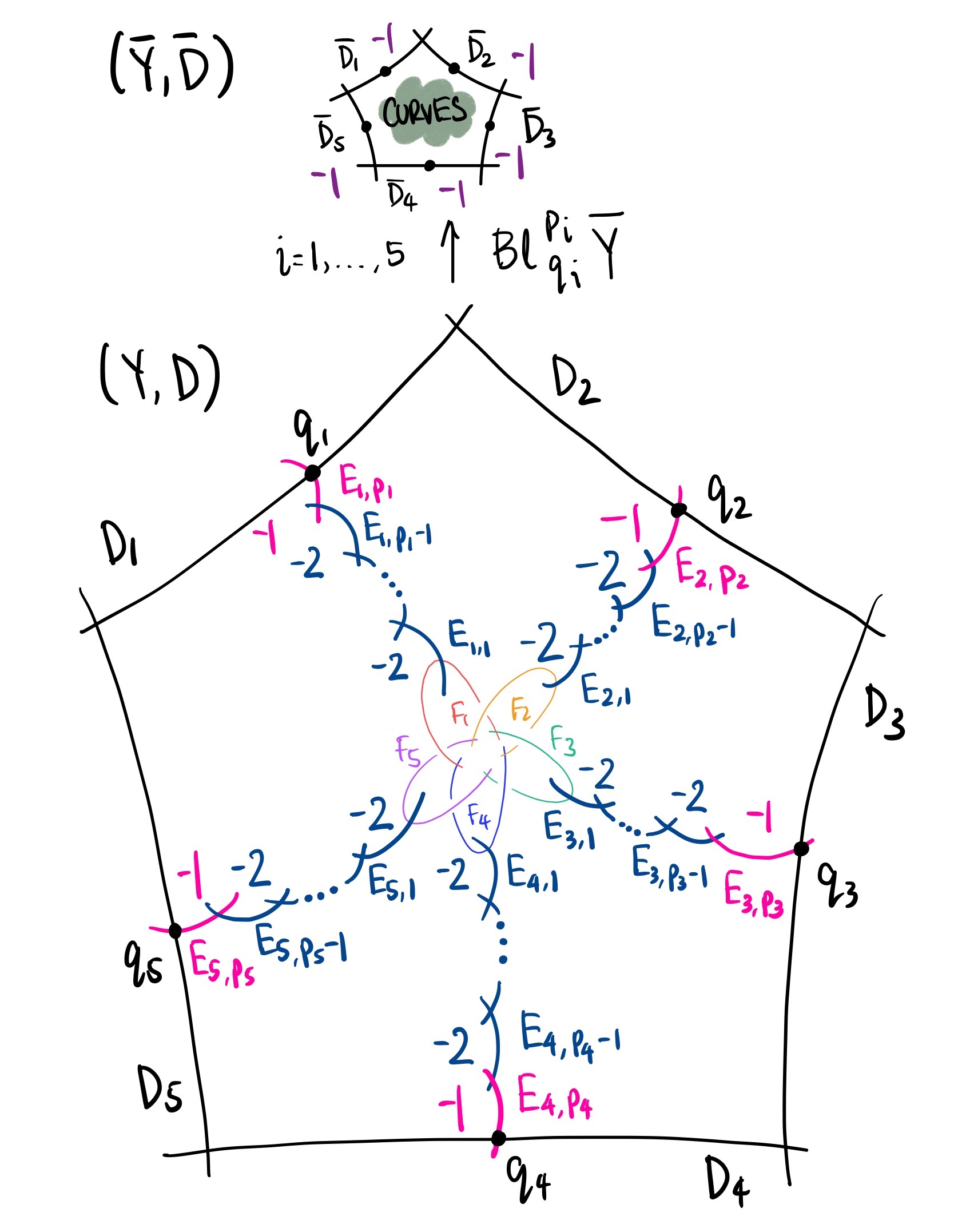}
\caption[The general picture for $n=5$]{Blowing up once at each point $q_{1}, \dots, q_{5}$ results in the diagram above.}
\label{fig:n5-pBUs}
\end{figure}

\begin{figure}
\centering
\includegraphics[scale=1.2]{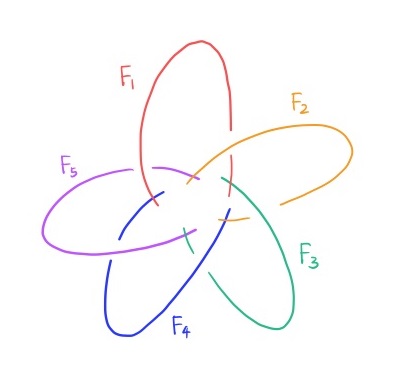}
\caption[Some central curves in the case $n=5$]{This drawing shows the curves in the center of case $n=5$.}
\label{fig:n5-pBUs}
\end{figure}

The dual elements $E_{i, j}^{\ast}$ and $F_{i}$, where $1 \leq i \leq 5$ and $1 \leq j \leq p_{i}$, are defined as follows:
\begin{align*}
E_{i, p_{i}}^{\ast} &= D_{i} \\
E_{i, p_{i}-1}^{\ast} &= D_{i} + E_{i, p_{i}} \\
E_{i, p_{i}-2}^{\ast} &= D_{i} + 2E_{i, p_{i}} + E_{i, p_{i}-1} \\
& \vdots \\
E_{i, 1}^{\ast} &= D_{i} + (p_{i}-1) E_{i, p_{i}} + (p_{i} - 2) E_{i, p_{i}-1} + \cdots + 2E_{i, 3} + E_{i, 2} \\
F_{i}^{\ast} &= D_{i} + p_{i}E_{i, p_{i}} + (p_{i}-1)E_{i, p_{i}-1} + \cdots + 2E_{i, 2} + E_{i,1}
\end{align*}

By Lemma \ref{lem:curvYgenerators}, we can describe the cone of curves as follows:
\begin{center}
$\oCurv(Y) = \langle D_i, E_{i, j}, F_{i} \; \vert \; 1 \leq i \leq 5 \text{ and } 1 \leq j \leq p_{i} \rangle_{\mathbb{R}_{\geq 0}}$.
\end{center}

\subsection*{Number of boundary components $n=6$.}
\label{subsec:n6}

\begin{theorem} (A. Simonetti, Ph.D. Thesis, 2021)
\label{thm:Simonetti}
Let $(Y, D)$ be a log Calabi-Yau surface with negative definite or negative semidefinite boundary of length $n=6$. Then $(Y, D)$ is obtained as a blowup of the toric surface $(\bar{Y}, \bar{D})$ with boundary a cycle of six $(-1)$-curves.
\end{theorem}

\begin{figure}
\centering
\includegraphics[scale=0.5]{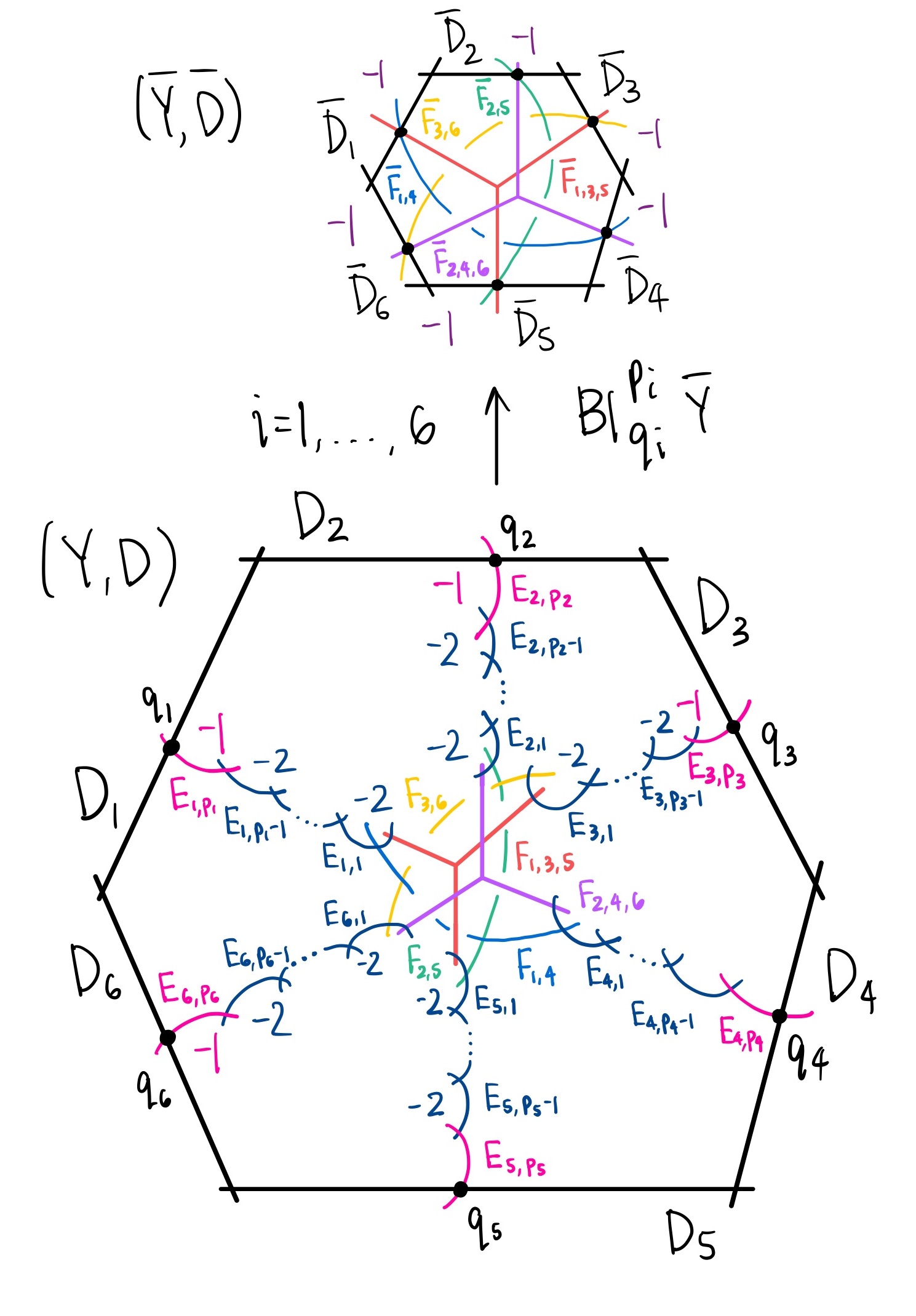}
\caption[The general picture for $n=6$]{This diagram shows the surface $(\bar{Y}, \bar{D})$ blown up at points $q_{i}$, each a total of $p_{i}$ times, for when $n = 6$.}
\label{fig:n6-pBUs}
\end{figure}

We have the toric surface $\bar{Y}$ with toric boundary $\bar{D} = \bar{D}_{1} + \cdots + \bar{D}_{6}$, a hexagon of $(-1)$-curves. Take $q_{i}$ where $i = 1 \dots, 6$ to be the points $(-1) \in \mathbb{C}^{\ast} \subset \mathbb{P}^{1} = \bar{D}_{i}$ for some choice of toric coordinates on $\bar{Y}$, and $Y$ the blowup of $\bar{Y}$ some number of times at each $q_{i}$. Define an index set as follows:
\begin{equation}
\label{eqn:n6-FindexSet}
K = \{\{1, 4\}, \{2, 5\}, \{3, 6\}, \{1, 3, 5\}, \{2, 4, 6\}\}.
\end{equation}
Let $F_{k}$ be the strict transform of $\bar{F}_{k}$ on $Y$. We consider the following five curves $\bar{F}_{k}$ on $\bar{Y}$ ($k \in K$), such that $\bar{F}_{k} \cap \bar{D} = \{q_{i} \; \vert \; i \in k\}$.
\begin{enumerate}
\item $\bar{F}_{1, 3, 5}$ and $\bar{F}_{2, 4, 6}$ are pullbacks of a line in $\mathbb{P}^{2}$ for two different birational morphisms $\bar{Y} \rightarrow \mathbb{P}^{2}$; and
\item $\bar{F}_{1, 4}$ and $\bar{F}_{2, 5}$ and $\bar{F}_{3, 6}$ are fibers of three different morphisms $\bar{Y} \rightarrow \mathbb{P}^{1}$.
\end{enumerate}

The classes of the curves $\{\bar{F}_{k}\}$ span $\Pic(\bar{Y})$ with one relation:
\begin{center}
$\bar{F}_{1, 4} + \bar{F}_{2, 5} + \bar{F}_{3, 6} = \bar{F}_{1, 3, 5} + \bar{F}_{2, 4, 6}$.
\end{center}

For $i = 1, \dots, 6$ and $j = 0, \dots, p_{i}$, define a divisor
\begin{center}
$A_{i, j} = D_{i} + (p_{i} - j)E_{i, p_{i}} + (p_{i} - j - 1)E_{i, p_{i} - 1} + \dots, + E_{i, j+1}$.
\end{center}

Then for $j > 0$ we have
\begin{equation}
\label{eqn:dualBasisOfT}
  A_{i, j} \cdot E_{s, t} =\left\{
  \begin{array}{@{}ll@{}}
    1 & \text{if } i = s \text{ and } j=t;  \\
    0 & \text{otherwise.}
  \end{array}\right.
\end{equation}
The set $S = \{E_{i, j} \; \vert \; i = 1, \dots, 6\} \cup \{F_{k} \; \vert \; k \in K\}$ spans $\Pic(Y)$ because the set $\{\bar{F}_{k} \; \vert \; k \in K\}$ spans $\Pic(\bar{Y})$.

Let $C \subset Y$ be an irreducible curve. Suppose that $C \neq D_{i} ,E_{i, j}$ for all $i, j$. Then $C \cdot A_{i, j} \geq 0$ for all $i, j$. We can write
\begin{center}
$C = \displaystyle{\sum} a_{i, j} E_{i, j} + \displaystyle{\sum} b_{k} F_{k} \in \Pic(Y)$
\end{center}
Computing the intersection numbers $A_{i, j} \cdot E_{s, t}$ and $A_{i, j} \cdot F_{k}$ results in the following inequalities:

\begin{align*}
a_{i, j} &\geq 0 \text{ for all } i, j; \\
b_{1, 4} &+ b_{1, 3, 5} \geq 0; \\
b_{1, 4} &+ b_{2, 4, 6} \geq 0; \\
b_{2, 5} &+ b_{1, 3, 5} \geq 0; \\
b_{2, 5} &+ b_{2, 4, 6} \geq 0; \\
b_{3, 6} &+ b_{1, 3, 5} \geq 0; \text{ and } \\
b_{3, 6} &+ b_{2, 4, 6} \geq 0; \\
\end{align*}

The last six inequalities define the cone
\begin{center}
$\sigma := \langle [F_{k}] \; \vert \; k \in K \rangle_{\mathbb{R} \geq 0} \subset V$,
\end{center}
where $V := \langle [F_{k}] \; \vert \; k \in K \rangle_{\mathbb{R}}$. Using the spanning set 
\begin{center}
$\{[F_{k}]\}$ where $k \in K = \{\{1, 4\}, \{2, 5\}, \{3, 6\}, \{1, 3, 5\}, \{2, 4, 6\}\}$
\end{center}
of $V$, we can identify $\sigma$ with the cone
\begin{center}
$\langle \bar{e}_{1}, \dots, \bar{e}_{5} \rangle_{\mathbb{R} \geq 0} \subset \mathbb{R}^{5} / \langle (1, 1, 1, -1, -1) \rangle_{\mathbb{R}}$.
\end{center}

Then, we may assume that $b_{k} \geq 0$ for all $k \in K$, so that $C$ lies in the cone generated by the $E_{i, j}$ and the $F_{k}$. Therefore $\Curv(Y) = \langle D_{i}, E_{i, j}, F_{k} \; \vert \; i = 1, \dots, 6 \text{ and } j = 1, \dots, p_{i} \text{ and } k \in K \rangle_{\mathbb{R}_{\geq 0}}$.

\begin{corollary}
\label{cor:newExamplesMDS}
A log Calabi-Yau surface $(Y_{e}, D_{e})$ with split mixed Hodge structure which has boundary $D_{e}$ consisting of no more than six components is an example of a Mori Dream Space.
\end{corollary}
\begin{proof}
This follows from \ref{eqn:PicYClY}, Theorem \ref{thm:proofn1-6}, Theorem \ref{nefImpliesSemiample}, and Definition \ref{def:MoriDreamSpace}.
\end{proof}

%
%

\section{MOTIVATIONS}
\label{Motivations}

Let $(q \in X)$ be a cusp singularity. Cusp singularities come in dual pairs such that the links are diffeomorphic but have opposite orientations (cf. \cite{L81}, \S III.2.1). There is the following conjecture (cf. \cite{E15}, Conjecture 6.2.5):
\begin{enumerate}
\item The smoothing components of the deformation space of $(q \in X)$, up to isomorphism, are in bijective correspondence with deformation types of log Calabi-Yau surfaces $(Y, D)$ such that $D$ does not contain any $(-1)$-curves and contracts to the dual cusp $p$.
\item The smoothing component of $(q \in X)$ associated to $(Y, D)$ is the {\it Looijenga space} determined by the action of $\Adm$ on the nef effective cone $\Nefe(Y^{\prime}_{gen})$, which is contained in $\langle D_{1}, \dots, D_{n} \rangle^{\bot} \otimes_{\mathbb{Z}} \mathbb{R}$. (Looijenga's construction is described in \cite{L03}, Section $4$).
\end{enumerate}

Looijenga's construction requires that $\Adm$ acts on $\Nefe(Y^{\prime}_{gen})$ with a rational polyhedral fundamental domain (Theorem \ref{coneConjectureYgen}), and this is one motivation for the cone conjecture for log Calabi-Yau surfaces. This is analogous to the original motivation for the Morrison cone conjecture for Calabi-Yau threefolds $Y$ (\cite{M93}): Morrison argues that the Looijenga space determined by the action of $\Aut(Y)$ on $\Nefe(Y)$ is identified with a neighborhood of a boundary point of the moduli of the mirror of $Y$.

The log Calabi-Yau cone conjecture can provide insight into the original Morrison cone conjecture because it includes more accessible cases. For instance, in every dimension there are many log Calabi-Yau pairs $(Y, D)$ such that the variety $Y$ is rational. In addition, the cone conjecture is related to the abundance conjecture (\cite{T11}), which is a long-standing open question of the minimal model program.

\begin{remark}
The cone conjecture for log Calabi-Yau surfaces suggests that the Morrison cone conjecture is false in general, because it is the monodromy group $\Adm$ that acts with a rational polyhedral fundamental domain on $\Nefe(Y_{gen})$, and not the automorphism group.
\end{remark}

\begin{remark}
The explicit description of $\Nef(Y_{e})$ can be used to verify the conjecture (1) stated above. For $n \leq 5$, this follows from work of Looijenga \cite{L81}, and for $n = 6$, we expect that it can be verified using work of Brohme (\cite{B95}). The deformation theory for $n > 6$ is not known.
\end{remark}


\begin{thebibliography}{GHK15}
\singlespacing
\bibitem[AV14]{AV14} E.~Amerik, M.Verbitsky, Morrison-Kawamata cone conjecture for hyperk\"{a}hler manifolds, Ann. Sci. Ec. Norm. Super. (4) 50 (2017), no. 4, 973-993.
\bibitem[B13]{B13} J.~Blanc, Symplectic birational transformations of the plane, Osaka J. Math.~50 (2013), no.~2, 573--590.
\bibitem[B95]{B95} S.~Brohme, Versal base spaces of minimally elliptic singularities, Abh. Math. Sem. Univ. Hamburg.~65 (1995), 175-187. 
\bibitem[CO15]{CO15} S.~Cantat and K.~Oguiso, Birational automorphism groups and the movable cone theorem for Calabi-Yau manifolds of Wehler type via universal Coxeter groups, Amer. J. Math.~137 (2015), no.~4, 1013--1044. 
\bibitem[CK16]{CK16} A.~Corti and A-S.~Kaloghiros, The Sarkisov program for Mori fibred Calabi-Yau pairs, Algebr. Geom.~3 (2016), no.~3, 370--384. 
\bibitem[D08]{D08} I.~Dolgachev, Reflection groups in algebraic geometry, Bull. Amer. Math. Soc. (N.S.)~45 (2008), no.~1, 1--60.
\bibitem[E15]{E15} P.~Engel, \href{https://academiccommons.columbia.edu/doi/10.7916/D8028QGQ}{A proof of Looijenga's conjecture via integral-affine geometry}, PhD thesis, Columbia University (2015).
\bibitem[EF16]{EF16} P.~Engel and R.~Friedman, Smoothings and Rational Double Point Adjacencies for Cusp Singularities, preprint arXiv:1609.08031 [math.AG] (2016), 85 pp.
\bibitem[F13]{F13} R.~Friedman, On the ample cone of a rational surface with an anticanonical cycle, Algebra Number Theory~7 (2013), no.~6, 1481--1504.
\bibitem[F15]{F15} R.~Friedman, On the geometry of anticanonical pairs, preprint arXiv:1502.02560v2 [math.AG] (2015), 86 pp.
\bibitem[F83]{F83} R.~Friedman, R.~Miranda, Smoothing cusp singularities of small length, Math. Ann.~263 (1983), no.~2, 185--212.
\bibitem[GHK15a]{GHK15a} M.~Gross, P.~Hacking, S.~Keel, Mirror symmetry for log Calabi-Yau surfaces I, Publ. Math. Inst. Hautes \'Etudes Sci.~122 (2015), 65--168. 
\bibitem[GHK15b]{GHK15b} M.~Gross, P.~Hacking, S.~Keel, Moduli of surfaces with an anti-canonical cycle, Compos. Math.~151 (2015), no.~2, 265--291. 
\bibitem[G62]{G62} H.~Grauert, \"{U}ber Modifikazionen und excepzionelle analytische Mengen, Math. Ann. 146 (1962), 331--368. 
\bibitem[H77]{H77} R.~Hartshorne, Algebraic Geometry, Springer Graduate Texts in Mathematics (1977).
\bibitem[HK00]{HK00} Y.~Hu and S.~Keel, Mori dream spaces and GIT, Michigan Math. J.~48 (2000), 331--348.
\bibitem[HKe21]{HKe21} H.~Hacking, A.~Keating, Homological mirror symmetry for log Calabi-Yau surfaces, preprint arXiv:2005.05010v2.
\bibitem[H16]{H16} D.~Huybrechts, Lectures on K3 surfaces, Cambridge Stud. Adv. Math.~158, C.U.P. (2016).
\bibitem[K97]{K97} Y.~Kawamata, On the cone of divisors of Calabi-Yau fiber spaces, Internat. J. Math.~8 (1997), no.~5, 665--687.
\bibitem[Ke15]{Ke15} A.~Keating, Homological mirror symmetry for hypersurface cusp singularities, preprint arXiv:1510.08911v2.
\bibitem[KM98]{KM98} J.~Koll\'ar and S.~Mori, Birational geometry of algebraic varieties, Cambridge Tracts in Math.~134, C.U.P. (1998).
\bibitem[K94]{K94} S.~Kov\'acs, The cone of curves of a $K3$ surface, Math. Ann.~300 (1994), no.~4, 681--691. 
\bibitem[L03]{L03} E.~Looijenga, Compactifications defined by arrangements II: Locally symmetric varieties of type IV, Duke Math. J. 119 (2003), 527-–588.
\bibitem[L81]{L81} E.~Looijenga, Rational surfaces with an anticanonical cycle, Ann. of Math.~(2) 114 (1981), no.~2, 267--322.
\bibitem[L14]{L14} E.~Looijenga, Discrete automorphism groups of convex cones of finite type, Compos. Math.~150 (2014), no.~11, 1939--1962.
\bibitem[LW86]{LW86} E.~Looijenga, J.~Wahl, Quadratic functions and smoothing surface singularities, Topology~25 (1986), no.~3, 261--291.
\bibitem[M90]{M90} L.~McEwan, Families of rational surfaces preserving a cusp singularity, Trans. Amer. Math. Soc.~321 (1990), no.~2, 691--716.
\bibitem[M93]{M93} D.~Morrison, Compactifications of moduli spaces inspired by mirror symmetry, Journ\'ees de G\'eom\'etrie Alg\'ebrique d'Orsay, Ast\'erisque No. 218 (1993), 243--271.
\bibitem[M61]{M61} D.~Mumford, The topology of normal singularities of an algebraic surface and a criterion for simplicity, Inst. Hautes \'Etudes Sci. Publ. Math.~9 (1961), 5--22.
\bibitem[S87]{S87} F.~Scattone, On the compactification of moduli spaces for algebraic K3 surfaces, Mem. Amer. Math. Soc.~70 (1987), no. ~374.
\bibitem[S11]{S11} B.~Simon, ``Convexity: An analytic viewpoint", C.U.P. (2011).
\bibitem[S85]{S85} H.~Sterk, Finiteness results for algebraic K3 surfaces, Math. Z. 189, 507–513 (1985).
\bibitem[S21]{S21} A.~Simonetti, \href{https://scholarworks.umass.edu/dissertations_2/2355/}{Equivariant smoothings of cusp singularities}, PhD thesis, University of Massachusetts, Amherst (2021).
\bibitem[T08]{T08} B.~Totaro, Hilbert's 14th problem over finite fields and a conjecture on the cone of curves, Compos. Math.~144 (2008), no.~5, 1176--1198.
\bibitem[T10]{T10} B.~Totaro, The cone conjecture for Calabi--Yau pairs in dimension 2, Duke Math. J.~154 (2010), no.~2, 241--263.
\bibitem[T11]{T11} B.~Totaro, Algebraic surfaces and hyperbolic geometry, Current Developments in Algebraic Geometry, MSRI publications, Vol. 59 (2011).
\bibitem[W88]{W88} J.~Wehler, $K3$ surfaces with Picard number $2$, Arch. Math. (Basel)~50 (1988), no.~1, 73--82.
\end{thebibliography}
\end{document}